\documentclass[reqno,10pt,centertags,draft]{amsart}
\usepackage{amsmath,amsthm,amscd,amssymb,latexsym,upref,stmaryrd}
\usepackage{color}

\makeatletter
\def\theequation{\@arabic\c@equation}

\allowdisplaybreaks
\numberwithin{equation}{section}

\newcommand{\re}{\operatorname{Re}}

\newcommand{\Ran}{\operatorname{ran}}

\renewcommand{\ker}{\operatorname{ker}}
\newcommand{\diag}{\operatorname{diag}}
\newcommand{\Span}{\operatorname{span}}

\newcommand{\RR}{{\mathbb{R}}}

\newcommand{\cL}{{\mathcal{L}}}
\newcommand{\cE}{{\mathcal{E}}}
\newcommand{\cB}{{\mathcal{B}}}
\newcommand{\cD}{{\mathcal{D}}}
\newcommand{\cM}{{\mathcal{M}}}
\newcommand{\D}{\mathrm{d}}
\newcommand{\dd}{\,\mathrm{d}}
\newcommand{\e}{\mathrm{e}}
\newcommand{\Sp}{\operatorname{Sp}}
\newcommand{\Spe}{\operatorname{Sp_{ess}}}

\newtheorem{theorem}{Theorem}[section]

\newtheorem{lemma}[theorem]{Lemma}

\newtheorem{hypothesis}[theorem]{Hypothesis}

\theoremstyle{remark}
\newtheorem{remark}[theorem]{Remark}

\newtheorem{example}[theorem]{Example}
\begin{document}

\title[Stable foliation]{Stable foliations near a traveling front for reaction diffusion systems} 

\date{\today}
\author{Yuri Latushkin, Roland Schnaubelt and Xinyao Yang}

\address{YL, Department of Mathematics,
University of Missouri, Columbia, MO 65211, USA}
\email{latushkiny@missouri.edu}
\urladdr{http://www.math.missouri.edu/personnel/faculty/latushkiny.html}

\address{RS, Department of Mathematics, Karlsruhe Institute of Technology,
76128 Karlsruhe, Germany}
\email{schnaubelt@kit.edu}
\urladdr{http://www.math.kit.edu/iana3/~schnaubelt/}

\address{XY, Department of Mathematics,
University of Missouri, Columbia, MO 65211, USA}
\email{xywp8@mail.mizzou.edu}

\thanks{Partially supported by the US National Science
Foundation under Grants NSF DMS-1067929, DMS-1311313, by the Research Board and Research Council 
of the University of Missouri, and by the Simons Foundation. This research was funded by the
IRSES program of the European Commission (PIRSES-GA-2012-318910). RS gratefully acknowledges financial support 
by the Deutsche Forschungsgemeinschaft (DFG) through CRC 1173.}

\begin{abstract}
We establish the existence of a stable foliation in the vicinity of a traveling front solution for systems 
of reaction diffusion equations in one space dimension that arise in the study of chemical reactions models and 
solid fuel combustion. In this way we complement the orbital stability results from earlier papers
by A.~Ghazaryan, S.~Schecter and Y.~Latushkin. The essential spectrum of the differential operator obtained 
by linearization at the front touches the imaginary axis. In spaces with exponential weights, one can shift 
the spectrum to the left. We study the nonlinear equation on the intersection of the unweighted and weighted 
spaces. Small translations of the front form a center unstable manifold. For each small translation we prove 
the existence of a stable manifold containing the translated front and show that the stable manifolds foliate 
a small ball centered at the front.
\end{abstract}

\maketitle
\section{Introduction}
Traveling fronts are solutions to partial differential equations which move with constant speed without 
changing their shapes and which are asymptotic to spatially constant steady states. Traveling fronts 
are important by many reasons and have intensively been studied. We refer to the books and review papers 
\cite{F,VVV,X} and to more recent sources such as \cite{KP,LW,RM1,RM2,RM3,Sa02,TZKS}
that contain further bibliography. 

In this paper we study the dynamics in the vicinity of traveling fronts for a class  of reaction 
diffusion equations in one space dimension. A typical example 
arising in combustion theory for solid fuels, cf.\ \cite{BLR,GLS1,MS}, is given by
\begin{equation}\label{eq:exa0}
u_t=u_{xx}+ug(v),\quad
v_t=\epsilon v_{xx}+\kappa ug(v),
\end{equation}
where $u,v\in\mathbb{R}$, $\epsilon\geq 0$,  $\kappa\in\mathbb{R}$,
and  $g(v)=\e^{-1/v}$ for $v>0$ and $g(v)=0$ for $v\leq 0$.
These and more general equations covered by our hypotheses often appear in the work on chemical 
reaction models and in combustion models, see, e.g., \cite{GSM,SMS,SKMS2,T,VV}. In such systems
the spectrum of  the linearization of the equation at the front touches the imaginary axis, 
cf.\ \cite{Sa02,SS}. To shift the spectrum to the left, one employs exponentially weighted spaces.
This idea  goes back to \cite{S} and \cite{PW}. However, in weighted spaces one can lose the Lipschitz 
properties of the nonlinearity.  We shall study  reaction terms with a certain "product" structure as 
in \eqref{eq:exa0} which allows one to overcome these difficulties.
The investigation of this class of nonlinearities was initiated by A.~Ghazaryan in 
\cite{G_indiana} and then continued in \cite{GLS,GLS1,GLS2}, see also the review paper \cite{GLS3}. In 
particular, it was proved in \cite{GLS2} that under appropriate assumptions on the nonlinearity the traveling 
front is orbitally stable; that is, any solution originating in a small vicinity of the front converges 
exponentially in the weighted norm to a translation of the front.

In this paper we continue the work in \cite{GLS2} now utilizing the theory of invariant manifolds, cf.\ 
\cite{BJ,CHT,Lu}. We analyze the dynamics in greater detail by proving in Theorem~\ref{STFOLTH}  
the existence of a stable foliation near the front. Specifically, we observe that the set of all translations 
of the front serves as a local central unstable manifold consisting of fixed points. Next, using the 
Lyapunov-Perron method, cf.\ e.g.\ \cite{LL,LPS1,LPS2}, we establish the existence and the fundamental 
properties of a locally invariant stable manifold going 
through each translation of the front. We also show that these manifolds foliate a small neighborhood of the 
front and therefore each point in the neighborhood belongs to one of them, cf.\ \cite{BLZ,CHT}. 
Moreover, the orbit of the point converges to the 
translation of the front along the stable manifold as proved in \cite{GLS2}.

In the construction of the local stable manifolds we have to face the problem that the linearization 
enjoys good decay properties only in weighted spaces on which the nonlinearity is  not locally Lipschitz.
To overcome this difficulty, we use both the product structure of the nonlinearity (cf.\ Hypothesis~\ref{HRFR})
and additional decay properties of the linearization at the limit of the traveling front as $\xi\to -\infty$,
see Lemmas~\ref{mainest}  and \ref{Lproper}.

The paper is organized as follows. In Section~\ref{sec1} we formulate our assumptions and prove several
preliminary results. In Section~\ref{sec2} we study the Lyapunov-Perron operator whose fixed points define 
 the stable manifolds. In Section~\ref{sec3} we formulate and prove our main result on the 
existence of the stable manifolds and discuss two examples.

{\bf Notation.}\, Throughout the paper,  $\left|\cdot\right|$ and $\left< \cdot,\cdot \right>$ are the 
Euclidean norm and the scalar product in $\mathbb{R}^n$. For a given map $f:\mathbb{R}^m\rightarrow\RR^k$,
its differential with respect to $y$ is written as 
$\partial_y f:\mathbb{R}^m\to \mathcal{B}(\mathbb{R}^m,\mathbb{R}^k)$. 
We let $\mathcal{B}(\mathcal{E},\mathcal{F})$ be the set of linear bounded operators between Banach spaces 
$\mathcal{E}$ and $\mathcal{F}$, and abbreviate $\mathcal{B}(\cE)=\mathcal{B}(\mathcal{E},\mathcal{E})$.
We denote by $C$ a generic constant that may change from one estimate to another, and use $T$ to designate 
transposition. For a Banach space with norm $\left\|\cdot\right\|$, we write 
$\mathbb{B}_{\delta}(\left\|\cdot\right\|)$ for the closed ball of radius $\delta$ centered at $0$.

We denote by $\mathcal{E}_0$ with norm $\left|\cdot \right|_0$ either the Sobolev space $H^1$ or the space 
$BUC$ of bounded uniformly continuous functions on $\mathbb{R}$ with vector values, and by
$\mathcal{E}_{\alpha}$ with norm $\left|\cdot\right|_{\alpha}$
the respective space of (exponentially) weighted functions, see \eqref{alphaomega}. 
Let $\left|\cdot\right|_{\beta}$ be the norm on the intersection space 
 $\mathcal{E}_{\beta}:=\mathcal{E}_0\cap\mathcal{E}_{\alpha}$; i.e., 
  $\left|y\right|_{\beta}:=\max\{ |y|_0,|y|_{\alpha} \}$. 

\section{The setting}\label{sec1}
We consider the system of reaction diffusion equations
\begin{equation}\label{eq4.2.1}
Y_t=DY_{xx}+R(Y),\qquad x\in\mathbb{R},\ t\geq0,
\end{equation}
where $D=\diag(d_1,\dots,d_n)$,  $d_j\geq 0$, $Y(t,x)\in \RR^n$, and $R: \mathbb{R}^n\rightarrow \mathbb{R}^n$ 
is a $C^3$ function satisfying  additional properties listed below. 

Passing in \eqref{eq4.2.1} to the moving coordinate frame $\xi=x-ct$ and redenoting $\xi$ again by $x$, 
we arrive at the nonlinear equation
\begin{equation}\label{eq4.2.3}
Y_t=DY_{xx}+cY_{x}+R(Y),\qquad  x\in\mathbb{R}, \ t\geq 0.
\end{equation}
We discuss the wellposedness of this system in Remark~\ref{DL}.

\begin{hypothesis}\label{HRF}
We assume that for some velocity $c\in\RR$ the system \eqref{eq4.2.3} admits a stationary solution 
$Y_0\in C^3(\RR)$; i.e, \eqref{eq4.2.1} possesses the traveling front solution $Y(t,x)=Y_0(x-ct)$. 
It is also required  that $Y_0(x)$ converges to the end states $Y_\pm$ as $x\rightarrow\pm\infty$
 exponentially; i.e.,
\begin{equation}\label{eq4.2.2}
\begin{aligned}
\left| Y_0(x)-Y_{-} \right|\leq C\e^{-\omega_-x},\qquad x\leq 0,\\
\left| Y_0(x)-Y_{+} \right|\leq C\e^{-\omega_+x},\qquad x\geq 0,
\end{aligned}
\end{equation}
for some $\omega_-<0<\omega_+$ and $C>0$. Replacing $R$ by $\tilde{R}(Y):=R(Y+Y_-)$, we can and will
assume that $Y_-=0$ (and we then drop the tilde).
\end{hypothesis}
 
We further assume that the nonlinear term $R$ in \eqref{eq4.2.1} and \eqref{eq4.2.3} has the following 
product structure.
 \begin{hypothesis}\label{HRFR}
The nonlinear term $R$ belongs to $C^3(\RR^n,\RR^n)$. In appropriate variables $Y=(U,V)^T$ with 
$U\in\mathbb{R}^{n_1}$, $V\in\mathbb{R}^{n_2}$ and $n_1+n_2=n$, we have 
\begin{equation}\label{eq4.3.1}
R(U,0)=(A_1U,0)
\end{equation}
for a constant $n_1\times n_1$ matrix $A_1$.
\end{hypothesis} 
In other words, we suppose that 
\begin{equation*}
R(U,V) =\begin{pmatrix} A_1U+ R_1(U,V)\\ R_2(U,V) \end{pmatrix},
\end{equation*}
where the maps $R_j$ belong to $C^2(\RR^n, \RR^{n_j})$ and satisfy $R_j(U,0)=0$ for $j\in\{1,2\}$ and 
$U\in \RR^{n_1}$. Note that condition \eqref{eq4.3.1} yields $R(0,0)=R(Y_-)=0$. We also split
\[D=\begin{pmatrix}
D_1 & 0\\0 & D_2
\end{pmatrix}, \qquad 
\text{where} \quad D_1=\diag(d_1,\dots,d_{n_1}), \quad D_2=\diag(d_{n_1+1},\dots,d_n).\]

Let $q\in\mathbb{R}$. We write $Y_q(x)=Y_0(x-q)$ for the shifted wave. Since \eqref{eq4.2.3} is 
translationally invariant, $Y_q$ is again a steady state solution of \eqref{eq4.2.3} and thus yields
a traveling wave solution for \eqref{eq4.2.1}. Linearizing \eqref{eq4.2.3} at $Y_q$ (that is, substituting 
$Y_q+Y$ instead of $Y$ in \eqref{eq4.2.3}), we arrive at the equation
\begin{equation}\label{eq4.3.3}
Y_t=L_qY+F_q(Y),\qquad \text{where} \quad
L_qY=DY_{xx}+cY_x+\partial_YR(Y_q)Y.
\end{equation}
Here,  $\partial_Y$ is the differential with respect to $Y\in\RR^n$ 
and the nonlinear term $F_q:\mathbb{R}^n\rightarrow \mathbb{R}^n$ is written as
\begin{equation}\label{eq4.3.4}
F_q(Y)=\int_{0}^{1}\left(\partial_YR(Y_q+tY)-\partial_YR(Y_q)\right)Y\dd t.
\end{equation}
The linearization of \eqref{eq4.2.3} at $Y_-=(0,0)^T$ is given by
\begin{equation}\label{eq4.3.5}
Y_t=L^-Y+G(Y),\qquad \text{where}\quad
L^-Y=DY_{xx}+cY_x+\partial_{Y}R(0)Y
\end{equation}
and $G:\mathbb{R}^n\rightarrow\mathbb{R}^n$;  $G(Y)=R(Y)-\partial_YR(0)Y$.
We remark that 
\begin{equation}\label{eq4.3.6}
(L_q-L^-)Y=B_qY\qquad \text{with}\quad
B_q(x)=\partial_{Y}R(Y_q(x))-\partial_YR(0).
\end{equation}
Below we impose conditions on $L_0$ at $q=0$; i.e., on the linearization at the original traveling 
wave $Y_0$. We further consider $L_q$ for $|q|\le q_0$ with some  $q_0>0$, which will be fixed sufficiently
small in the final theorem. The shifted wave $Y_q$ decays as in
 Hypothesis~\ref{HRFR} with the same exponents $\omega_\pm$ and constants $C$ only depending on $q_0$.
Assumption \eqref{eq4.3.1} also yields the formulas
\begin{equation}\label{eq4.4.1}
\partial_YR(0,0)=\begin{pmatrix}
A_1 & \partial_VR_1(0,0)\\ 0 & \partial_VR_2(0,0)
\end{pmatrix},\qquad 
L^-=\begin{pmatrix}
L^{(1)} & \partial_VR_1(0,0)\\
0 & L^{(2)}
\end{pmatrix}
\end{equation}
with the differential expressions
\begin{equation}\label{eq4.4.2}
\begin{aligned}
& L^{(1)}U=D_1U_{xx}+cU_x+A_1U,\\
& L^{(2)}V=D_2V_{xx}+cV_{x}+\partial_VR_2(0,0)V.
\end{aligned}
\end{equation}

\begin{remark}\label{DL}
We consider the equations \eqref{eq4.2.3} and \eqref{eq4.3.3} on the  space $\mathcal{E}_0$
which is either the Sobolev space $H^1(\mathbb{R})^n$ or the space of bounded uniformly continuous 
functions $BUC(\mathbb{R})^n$. It is straightforward to check that the nonlinearites $R$ and $F_q$ are 
Lipschitz on bounded subsets of $\cE_0$.

For the differential expressions $L_q$ and $L^-$ defined in \eqref{eq4.3.3} 
and \eqref{eq4.3.5}, respectively, 
we denote by $\mathcal{L}_{q}$ and $\mathcal{L}^{-}$ the differential operators on $\mathcal{E}_0$ on their 
natural domain $\cD$ defined as follows. For $\mathcal{E}_0=H^1(\mathbb{R})^n$, the domain $\cD$ 
of $\mathcal{L}_{q}$ and of $\mathcal{L}^{-}$ 
consists of the vector functions $Y=(Y_j)_{j=1}^n$ whose components $Y_j$ belong to $H^3(\mathbb{R})$ if 
$d_j>0$ and to $H^2(\mathbb{R})$ if $d_j=0$. For $\mathcal{E}_0=BUC(\mathbb{R})^n$, we choose the domain 
analogously with $H^3(\mathbb{R})$ replaced by $BUC^2(\mathbb{R})$ and $H^2(\mathbb{R})$ replaced by 
$BUC^1(\mathbb{R})$, the spaces of differentiable functions which are bounded and have bounded, 
uniformly continuous derivatives. The operators $\mathcal{L}_{q}$ and  $\mathcal{L}^{-}$ generate 
strongly continuous semigroups $\{T_q(t)\}_{t\ge0}$ and $\{S(t)\}_{t\ge0}$ on $\cE_0$, respectively,  
cf.\ e.g.\  \cite[\S 2.2]{GLS}. 

Standard results then show the local wellposedness of \eqref{eq4.3.3} in $\cE_0$ for initial values $y_0$
in the domain of $\mathcal{L}_{q}$, where the (classical) solutions belong to $C^1([0,t_0),\cE_0)$ and take
values in $\cD$. They are given by Duhamel's formula 
\begin{equation}
\label{VOCF0}
Y(t)=T_q(t)y_0+\int_0^tT_q(t-\tau)F_q(Y(\tau))\dd\tau,\qquad t\geq 0.
\end{equation}
 See e.g.\
Theorems~6.1.4 and 6.1.6 in \cite{pazy}. A function $Y\in C([0,t_0),\cE_0)$ satisfying \eqref{VOCF0}
is called a \emph{mild solution} of  \eqref{eq4.3.3}. This concept is strictly weaker than that of classical 
solvability. We mostly work with mild solutions. Similar remarks apply to \eqref{eq4.2.3} and 
the differential expression $D\partial_{xx} + c\partial_x$ equipped the same domain $\cD$. 
Approximating a given initial value $y_0\in \cE_0$ in $\cE_0$ by functions in $\cD$, we  see that
all mild solutions of \eqref{eq4.2.3} are given by $Y_q+Y(t)$ where $Y(t)$ solves \eqref{VOCF0}.   
\hfill$\Diamond$
\end{remark}

Let $\alpha=(\alpha_-,\alpha_+)\in\mathbb{R}^2$. We say that $\gamma_{\alpha}:\mathbb{R}\rightarrow\mathbb{R}$ 
is a weight function of class $\alpha$ if $\gamma_{\alpha}$ is $C^2$, $\gamma_{\alpha}(x)>0$ for all 
$x\in\mathbb{R}$, and $\gamma_{\alpha}(x)=\e^{\alpha_-x}$ for $x\leq -x_0$ and 
$\gamma_{\alpha}(x)=\e^{\alpha_+x}$ for $x\geq x_0$ for some  $x_0>0$. We shall always assume that 
\begin{equation}\label{alphaomega}
0<\alpha_-<-\omega_-\,\text{ and }\,0\leq \alpha_+<\omega_+,
\end{equation}
where $\omega_{\pm}$ are the exponents mentioned in \eqref{eq4.2.2}. Given such a pair
$\alpha=(\alpha_-,\alpha_+)$, we introduce the weighted space 
$\mathcal{E}_{\alpha}=\{ u:\RR\to\RR^n:\gamma_{\alpha }u\in\mathcal{E}_0 \}$ with the norm 
$|u|_{\alpha}=|\gamma_{\alpha}u|_0$. (Recall that $\mathcal{E}_0$ with norm $|\cdot|_0$ is either 
$H^1(\mathbb{R})^n$ or $BUC(\mathbb{R})^n$.) The intersection space
$\mathcal{E}_{\beta}=\mathcal{E}_{0}\cap\mathcal{E}_{\alpha}$ is endowed with the norm 
$|u|_{\beta}=\max\{ |u|_0,|u|_{\alpha} \}$. The differential expressions $L_q$, $L^-$ etc.\ equipped
with their natural domains define operators in $\mathcal{E}_{\alpha}$ which are denoted by 
$\mathcal{L}_{q,\alpha}$, $\mathcal{L}_{\alpha}^-$ etc.\ (cf.\ Remark~\ref{DL}). 
On the spectrum of $\mathcal{L}_{0,\alpha}$, we impose the following  assumptions.

\begin{hypothesis}\label{HypSpL} In addition to Hypotheses~\ref{HRF} and \ref{HRFR}, we assume that 
there exists $\alpha=(\alpha_-,\alpha_+)\in\mathbb{R}^2$ such that \eqref{alphaomega} with $\omega_{\pm}$ 
from \eqref{eq4.2.2} and the following assertions hold.
\begin{itemize}
\item[(a)] $\sup\{ \re\lambda:\lambda\in \Spe(\mathcal{L}_{0,\alpha})\}<0$ for the differential expression 
            $L_0$ defined in \eqref{eq4.3.3}.
\item[(b)] The only element of $\Sp(\mathcal{L}_{0,\alpha})$ in $\{ \lambda\in\mathbb{C}:\re\lambda\geq 0 \}$  
          is a simple eigenvalue at $\lambda=0$ with $Y_0'$ being the respective eigenfunction.
\end{itemize}
\end{hypothesis}

Here the essential spectrum $\Spe(A)$ of a closed densely defined operator contains all points in the spectrum
$\Sp(A)$ which are not isolated eigenvalues of finite algebraic multiplicity.
We discuss various consequences of the above  hypothesis which are important for our proofs. 

\begin{remark}\label{L2H1}
We claim that assertions (a) and (b) in Hypothesis \ref{HypSpL} are satisfied for 
$\mathcal{E}_0=H^1(\RR)^n$ or $\mathcal{E}_0=BUC(\RR)^n$ if and only if they hold when $\mathcal{E}_0$ 
is replaced by the space $L_2(\RR)^n$ and $\mathcal{E}_\alpha$ by the space $L^2_\alpha(\RR)^n$ 
of functions $u$ with $\gamma_\alpha u\in L^2(\RR)$ which is endowed with the norm 
$|u|_\alpha=|\gamma_\alpha u|_{L^2}$. 

Indeed, the ``if'' part of the claim above is proved in Lemma~3.8 of \cite{GLS2}. So we assume 
Hypothesis~\ref{HypSpL} for $\mathcal{E}_0=H^1(\RR)^n$ or $\mathcal{E}_0=BUC(\RR)^n$. Then assertion (a) 
of this hypothesis for $\mathcal{E}_0=L^2(\RR)^n$  is true since the right-hand boundary of the 
essential spectra of $\mathcal{L}_{0,\alpha}$ is the same for all three spaces by \cite[Lemma 3.5]{GLS2}. 
To show assertion (b) for $\mathcal{E}_0=L^2(\RR)^n$, we assume that $\mathcal{L}_{0,\alpha}$ on 
$L^2_\alpha(\RR)^n$ has an isolated eigenvalue $\lambda$ of finite algebraic multiplicity with 
$\re\lambda\ge0$. By means of the isomorphism $u(\cdot)\mapsto \gamma(\cdot)u(\cdot)$ between 
$L^2_\alpha(\RR)^n$ and $L^2(\RR)^n$ we obtain a differential operator $\hat{\mathcal{L}}$ in $L^2(\RR)^n$ 
which is similar to $\mathcal{L}_{0,\alpha}$ in $L^2_\alpha(\RR)^n$, cf.\ \cite[Eqn.~(3.2)]{GLS2}, 
and hence possesses the unstable isolated eigenvalue $\lambda$, too. Palmer's Dichotomy Theorem in \cite{Palm} 
says that the first order system corresponding to the second order eigenvalue problem for $\hat{\mathcal{L}}$ 
admits exponential dichotomies on $\RR_-$ and $\RR_+$. Arguing  as in the proof of Lemma~3.8
of \cite{GLS2}, we see that the respective eigenfunction $Z$ decays exponentially as $x\to\pm\infty$.
It thus belongs to $BUC(\RR)^n$, and also to $H^1(\RR)^n$ since $Z_x$ can be bounded by $Z$ itself due to
the eigenvalue equation, see (3.3) in \cite{GLS2}. As a result,  $\hat{\mathcal{L}}$ in
$H^1(\RR)^n$ or $BUC(\RR)^n$ has the unstable eigenvalue $\lambda$ and therefore also $\mathcal{L}_{0,\alpha}$ 
 in $\mathcal{E}_\alpha$. Hypothesis~\ref{HypSpL} now shows that $\lambda=0$, completing the proof of 
 the claim. \hfill$\Diamond$\end{remark}

\begin{lemma}\label{SPq}
Assume that Hypothesis~\ref{HypSpL} holds. Then assertions (a) and (b) in Hypothesis~\ref{HypSpL} are 
satisfied by the operator $\mathcal{L}_{q,\alpha}$ instead of $\mathcal{L}_{0,\alpha}$  and  by the 
function $Y_q'$ instead of $Y_0'$.
\end{lemma}

\begin{proof} 
The operators $\mathcal{L}_{q,\alpha}$ and $\mathcal{L}_{0,\alpha}$ are similar via the transformation 
$Y\mapsto Y(\cdot-q)$ which also maps $Y'$ into $Y_q'$. The assertions then easily follow.
\end{proof}

\begin{remark}\label{Projq}
Assume Hypothesis~\ref{HypSpL}. Lemma~\ref{SPq} says that $\lambda=0$ is an isolated simple eigenvalue for 
$\mathcal{L}_{q,\alpha}$. We let $P_q^c$ denote the spectral projection for $\mathcal{L}_{q,\alpha}$ in 
$\mathcal{E}_{\alpha}$ onto $\ker\mathcal{L}_{q,\alpha}=\Span\{Y_q'\}$.
Basic operator theory (see, e.g., \cite[Lemma 2.13]{DL}) yields that 
$$\Ran(I_{\mathcal{E}_{\alpha}}-P_q^c)=\ker P_q^c=\Ran(\mathcal{L}_{q,\alpha}).$$ 
Moreover, the one-dimensional projection $P_q^c$ is given by  \begin{equation}\label{PRF}
P_q^cY=\zeta_q(Y)\, Y_q',\qquad \zeta_q(Y_q')=1,
\end{equation}
for an element $\zeta_q$ in $\ker \mathcal{L}_{q,\alpha}^*$ which is also one dimensional, cf.\ 
\cite[Theorem~IV.5.13]{Kato}. As in the proof of Lemma \ref{SPq}, the operators $\mathcal{L}_{q,\alpha}^*$ and 
$\mathcal{L}_{0,\alpha}^*$ are similar and therefore the norms of $\zeta_q\in \mathcal{E}_{\alpha}^*$ are 
bounded uniformly for $|q|\le q_0$. Also, in view of Lemma~3.3 in \cite{GLS2}, the first three derivatives of 
the shifted wave $Y_q$ are bounded by $C\e^{-\omega_-\xi}$ for $\xi\le 0$ and by $C\e^{-\omega_+\xi}$ 
for $\xi\ge 0$ with  $\omega_\pm$ from Hypothesis~\ref{HRFR} and constants $C$ only depending on $q_0$.
We conclude that
\begin{align*}
|P_q^cY|_\alpha&= |\zeta_q(Y)|\,|Y_q'|_\alpha\le C|Y|_\alpha\,|Y_q'|_\alpha\le C|Y|_\beta\, |Y_q'|_\alpha,\\
|P_q^cY|_0&= |\zeta_q(Y)|\,|Y_q'|_0\le C|Y|_\alpha\,|Y_q'|_0\le C|Y|_\beta \,|Y_q'|_0.
\end{align*}
As a consequence, $P^c_q$ induces maps
\[ P_q^c \in \cB(\cE_\alpha)\cap  \cB(\cE_\beta,\cE_\alpha) \cap \cB(\cE_\alpha,\cE_\beta)
        \cap \cB(\cE_\beta)\cap \cB(\cE_\alpha,\cE_0) \cap  \cB(\cE_\beta,\cE_0)\]
 The complementary projection $P^s_q=I-P_q^c$ thus satisfies 
\[ P_q^s \in \cB(\cE_\alpha)\cap  \cB(\cE_\beta) \cap \cB(\cE_\beta,\cE_\alpha)\cap \cB(\cE_\beta,\cE_0).\]
We use the same notation $P_q^c$ and $P_q^s$ on all these spaces and their 
 norms are uniformly bounded for $|q|\le q_0$.
 The projections further satisfy 
\begin{equation}\label{eq:proj-diff}
\|P_q^c-P_p^c\|_{\mathcal{B}(\mathcal{E}_{\beta})}\leq C|q-p|, \qquad 
\|P_q^c-P_p^c\|_{\mathcal{B}(\mathcal{E}_{\alpha})}\leq C|q-p|
\end{equation}
for $|p|,|q|\leq q_0$  and a constant independent of $p$ and $q$.
In fact,  \eqref{eq4.3.3} yields
\begin{align}
L_q-L_p&=\partial_{Y}R(Y_0(\cdot-q))-\partial_{Y}R(Y_0(\cdot-p))\notag\\
&=\int_0^1\partial_{YY}R(sY_q+(1-s)Y_p)\dd s\;[Y_0(\cdot-q)-Y_0(\cdot-p)],\notag\\
Y_0(x-q)-Y_0(x-p)&=-\int_0^1Y_0'(x-p-s(q-p))(q-p)\dd s. \label{y-diff}
\end{align}
For $\cE_0=BUC$ we deduce
\begin{equation}\label{diffL}\begin{split}
\|\mathcal{L}_{q,\alpha}-\mathcal{L}_{p,\alpha}\|_{\mathcal{B}(\mathcal{E}_{\alpha})}
   &=\sup_{x\in\mathbb{R}}|\partial_{Y}R(Y_0(x-q))-\partial_{Y}R(Y_0(x-p))|
\leq C|q-p|,\\
\|\mathcal{L}_{q}-\mathcal{L}_{p}\|_{\mathcal{B}(\mathcal{E}_{0})}
   &=\sup_{x\in\mathbb{R}}|\partial_{Y}R(Y_0(x-q))-\partial_{Y}R(Y_0(x-p))|
\leq C|q-p|,
\end{split}\end{equation}
and similarly for $\cE_0=H^1$. These estimates can easily be transferred to the resolvents on 
a sufficiently small circle around 0  which implies the claim \eqref{eq:proj-diff}. \hfill$\Diamond$
\end{remark}

\begin{remark}\label{ProjSol}
To provide extra information, we now determine $\zeta_q$ from \eqref{PRF} as a solution of a differential 
equation. Remark~\ref{L2H1} yields that Hypothesis~\ref{HypSpL} is also true  if we replace $\cE_0$ by 
$L^2(\RR)$. We first determine $\zeta_q$ for the operator $\mathcal{L}_{q,\alpha}^*$ acting on the 
dual $L^2_\alpha(\RR)^*$ of the space $L^2_\alpha(\RR)$ of functions with the exponential weight 
$\gamma_\alpha$. We recall that the operator $\gamma_\alpha: L^2_\alpha(\RR)\to L^2(\RR);$
$Y(\cdot)\mapsto \gamma_\alpha(\cdot) Y(\cdot)$, is an isometric isomorphism.
Moreover, $L^2_\alpha(\RR)^*$ can be identified with $L^2$-space with the weight $1/\gamma_\alpha$, 
where the duality map between $L^2_\alpha(\RR)$ and $L^2_\alpha(\RR)^*$ is given by the usual (real) 
$L^2$-scalar product. Hence, the adjoint  operator $\gamma_\alpha^*:L^2(\RR)\to L^2_\alpha(\RR)^*$ 
coincides with the multiplication  operator by $\gamma_\alpha$.

The operator $\gamma_\alpha\mathcal{L}_{q,\alpha}\gamma_\alpha^{-1}$ in $L^2(\RR)$ is 
Fredholm since it is similar to the Fredholm operator $\mathcal{L}_{q,\alpha}$ in $L^2_\alpha(\RR)$.  
The adjoint of $\gamma_\alpha\mathcal{L}_{q,\alpha}\gamma_\alpha^{-1}$ 
in $L^2(\RR)$ is also Fredholm, and it is equal to 
 $\gamma_\alpha^{-1}\mathcal{L}_{q,\alpha}^*\gamma_\alpha$ since $\gamma_\alpha^*=\gamma_\alpha$. 
 We note that the dimension of the kernels is preserved by  similarity and duality. The
 functional $\zeta_q\in\ker\mathcal{L}_{q,\alpha}^*$ from \eqref{PRF} is then represented by 
 $\zeta_q=\gamma_\alpha Z_q$ where $Z_q\in L^2(\RR)$ belongs to 
 $\ker\big(\gamma_\alpha^{-1}\mathcal{L}_{q,\alpha}^*\gamma_\alpha\big)$. In other words, 
 $Z_q\in L^2(\RR)$ is the unique (up to a normalization) solution on $\RR$ of the differential equation 
 $\big(\gamma_\alpha^{-1}\mathcal{L}_{q,\alpha}^*\gamma_\alpha\big)Z_q=0$. Reasoning as in the proof of 
 Lemma~3.8 in \cite{GLS2} or in Remark~\ref{L2H1}
we conclude that 
 the solution $Z_q$ decays exponentially to zero as $x\to\pm\infty$.
 Moreover, $Z_q$ is  the translation $Z_0(\cdot-q)$  of $Z_0$, and the decay of the function $Z_q$ is 
 thus uniform in $q$ for $|q|\le q_0$. Formula \eqref{PRF} now yields
 \begin{equation}\label{piform}
  P_q^cY =\pi_q(Y)Y_q' \qquad\text{with}\quad  
  \pi_q(Y)=\int_{\mathbb{R}}\langle Z_q(x),\gamma_\alpha(x) Y(x)\rangle \dd x
 \end{equation}
for all $Y\in L^2_\alpha(\RR)$, where $Z_q$ is the exponentially decaying function normalized 
such that $\pi_q(Y_q')=1$.

Finally, returning to the cases $\mathcal{E}_0=H^1(\RR)^n$ or $\mathcal{E}_0=BUC(\RR)^n$, we notice that 
$\pi_q(\cdot)$ is a bounded functional on $\mathcal{E}_\alpha$ in both cases. Using also the decay properties
of $Y_q'$ recalled in Remark~\ref{Projq},
we confirm from \eqref{piform} once again that $P_q^c$ is a bounded operator from both  
$\mathcal{E}_\beta$ and $\mathcal{E}_\alpha$ into $\mathcal{E}_\beta$, with uniform constants 
for $q\in [-q_0,q_0]$. \hfill$\Diamond$
\end{remark}

\begin{remark}\label{RB}
Let $B_q$ be multiplication operator induced by the matrix valued function $B_q(\cdot)$ from
\eqref{eq4.3.6}. Lemma~8.2 of \cite{GLS2} says that $B_q$ belongs to $\mathcal{B}(\cE_{\alpha},\cE_0)$.
 As in assertion (3) of this lemma, one also sees that 
 $\|B_q-B_p\|_{\mathcal{B}(\cE_{\alpha},\cE_0)}\le C|q-p|$ for $q,p\in[-q_0,q_0]$. Inspecting the proofs, 
 we see that the constants do not depend on $p$ and $q$, but on $q_0$.
\hfill$\Diamond$
\end{remark}

The operators $\mathcal{L}_q$ and $\mathcal{L}_{q,\alpha}$ generate strongly continuous semigroups
on $\mathcal{E}_0$ and $\mathcal{E}_\alpha$, respectively, which are both denoted by
$\{ T_q(t)\}_{t\geq 0}$, see e.g.\ \cite[\S 2.2]{GLS}. By Lemma~\ref{SPq}, there are numbers
\[0>-\nu> \sup\{\re\lambda:\lambda \in \Sp(\mathcal{L}_{q,\alpha})\setminus \{0\}\},\]
 Lemma~3.13 of \cite{GLS2} then yields the exponential decay
\begin{equation}\label{exp-stab}
\|T_q(t)P_q^s\|_{\mathcal{B}(\mathcal{E}_{\alpha})}\leq C\e^{-\nu t},\qquad t\ge0,
\end{equation}
see also \cite{GLS}. The constant $C$ can be chosen unform in $q$ because
of the transformation used in the proof of Lemma~\ref{SPq}.

Also the operators $\cL^-$ and $\cL^-_\alpha$ generate  strongly continuous semigroups
on $\mathcal{E}_0$ and $\mathcal{E}_\alpha$, designated by $\{ S(t)\}_{t\geq 0}$.
Since the multiplication operator $B_q$ is bounded on these spaces,  formula \eqref{eq4.3.6} implies
 the variation of constant formula
\begin{equation}\label{eq4.5.1}
T_q(t-\tau)=S(t-\tau)+\int_{\tau}^{t}S(t-s)B_qT_q(s-\tau)\dd s,\qquad t\geq\tau\geq0,\ q\in\mathbb{R}.
\end{equation}
 The upper triangular structure of the operator $\mathcal{L}^-$ indicated in \eqref{eq4.4.1} 
implies an analogous representation of the semigroup
 \begin{equation}\label{eq4.5.2}
S(t)=\begin{pmatrix}
S_1(t) & Q(t)\\0 & S_2(t)
\end{pmatrix} \quad\text{ and } \quad Q(t)=\int_0^t S_1(t-s)\partial_VR_1(0,0)S_2(s)\dd s.
\end{equation}
Here $\{ S_1(t) \}_{t\geq 0}$ and $\{ S_2(t) \}_{t\geq 0}$ are the semigroups generated by the operators $
\mathcal{L}^{(1)}$ and $\mathcal{L}^{(2)}$ from \eqref{eq4.4.2}, respectively. On these semigroups
we impose the following assumptions.

\begin{hypothesis}\label{hypos}
The strongly continuous semigroup $\{S_1(t)\}_{t\geq 0}$ is bounded and the semigroup 
$\{S_2(t)\}_{t\geq 0}$ is uniformly exponentially stable on $\mathcal{E}_0$; that is,
\begin{equation*}
\left\| S_1(t) \right\|_{\mathcal{B}(\mathcal{E}_0)}\leq C, \qquad 
\left\| S_2(t) \right\|_{\mathcal{B}(\mathcal{E}_0)} \leq C\e^{-\rho t}
\end{equation*}
for some $\rho>0$ and all $t\geq 0$.
\end{hypothesis}

Hypothesis~\ref{hypos} and \eqref{eq4.5.2} imply the boundedness of $\{ S(t) \}_{t\geq 0}$ on $\cE_0$; i.e.,
\begin{equation}\label{eq4.5.4}
\left\| S(t) \right\|_{\mathcal{B}(\mathcal{E}_0)}\leq C,\qquad t\geq0.
\end{equation} 

We next show that the semigroup $\{T_q(t)\}_{t\geq 0}$ is bounded on the space $\mathcal{E}_\beta$, too
\begin{lemma}\label{lemBT}
Assume Hypotheses~\ref{HypSpL} and \ref{hypos}. Take $q_0>0$  and let $\alpha=(\alpha_-,\alpha_+)$ 
satisfy \eqref{alphaomega}. Then we have
\begin{equation}\label{BTalpha}
\sup\limits_{|q|\leq q_0}\sup\limits_{t\geq 0}\|T_q(t)\|_{\mathcal{B}(\mathcal{E}_{\beta})}<\infty.
\end{equation}
\end{lemma}
\begin{proof}
 The variation of constant formula \eqref{eq4.5.1} yields on $\mathcal{E}_\beta$
\begin{equation} \label{TSST}
T_q(t)P_q^s=S(t)P_q^s+\int_{0}^{t}S(t-s)B_qT_q(s)P_q^s\dd s.
\end{equation}
As noted in Remark~\ref{Projq} and \eqref{eq4.5.4}, the projection $P_q^s$ belongs to 
$\mathcal{B}(\mathcal{E}_\beta, \mathcal{E}_0)$  and to $\mathcal{B}(\mathcal{E}_\beta, \mathcal{E}_\alpha)$ 
while the semigroup $S(t)$ is uniformly bounded in $\cE_0$ for $|q|\le q_0$ and $t\ge0$, respectively. 
Using \eqref{TSST}, these facts, Remark~\ref{RB} and the exponential decay in \eqref{exp-stab}, we can estimate
\begin{equation*}
\begin{aligned}
\|T_q(t)P_q^s\|_{\mathcal{B}(\mathcal{E}_\beta, \mathcal{E}_0)}
 &\leq C\|P_q^s\|_{\mathcal{B}(\mathcal{E}_\beta, \mathcal{E}_0)}\\
   &\qquad+C\int_0^t\|B_q\|_{\mathcal{B}(\cE_{\alpha},\cE_0)}\|T_q(s)P_q^s\|_{\mathcal{B} (\cE_{\alpha})} 
             \|P_q^s\|_{\mathcal{B}(\mathcal{E}_\beta, \mathcal{E}_\alpha)}\dd s\\
   &\leq C+C\int_0^t \e^{-\nu t}\dd s\leq C
\end{aligned}
\end{equation*}
for all $t\geq0$ and $|q|\leq q_0$, with uniform constants. In view of the inequality 
$\|T_q(t)P_q^s\|_{\mathcal{B}(\mathcal{E}_{\alpha})}\leq C\e^{-\nu t}$ from \eqref{exp-stab}, we have proved
\eqref{BTalpha} with $T_q(t)$ replaced by $T_q(t)P_q^s$. Writing the semigroup as
$T_q(t)=T_q(t)P_q^s+T_q(t)P_q^c$ on $\mathcal{E}_\beta$, it remains to show \eqref{BTalpha} 
with $T_q(t)$ replaced by $T_q(t)P_q^c$. Recall from Remark~\ref{Projq} that 
$P_q^c=I-P_q^s\in\mathcal{B}(\mathcal{E}_\beta)$ projects $\mathcal{E}_\beta$ onto the kernel of the 
generators $\mathcal{L}_{q,0}$ and $\mathcal{L}_{q,\alpha}$ of the semigroup $\{ T_q(t)\}_{t\geq 0}$ on 
$\mathcal{E}_0$ and $\mathcal{E}_\alpha$. We conclude that  $T_q(t)P_q^cY=P_q^cY$ for all 
$Y\in\mathcal{E}_\beta$ and $t\ge0$. Therefore, $\|T_q(t)P_q^c\|_{\mathcal{B}(\cE_\beta)}\le C$  for  
$t\geq 0$, completing the proof of \eqref{BTalpha}.
\end{proof}

\section{The Lyapunov-Perron operator}\label{sec2}
In this section we introduce the Lyapunov-Perron operator associated with the nonlinear equation 
\eqref{eq4.3.3} and show that it is a contraction of a small ball in a certain space of functions
$u:\RR\to\cE_0\cap\cE_\alpha$. First, we establish the 
main technical estimates for the nonlinearity $F_q:\RR^n\rightarrow\mathbb{R}^n$ defined in \eqref{eq4.3.4}.

\begin{lemma}\label{mainest}
Assume that $\alpha=(\alpha_-,\alpha_+)$ satisfies \eqref{alphaomega} and that the nonlinearity 
$R\in C^3(\mathbb{R}^n,\mathbb{R}^n)$ fulfills \eqref{eq4.3.1}. Let $\delta_1>0$ and choose a radius 
$\delta\in(0,\delta_1]$. Then for all functions $y=(u,v)$ and $\bar{y}=(\bar{u},\bar{v})$  from 
$\mathcal{E}_\beta$ with $|y|_\beta,|\bar{y}|_\beta \le\delta$ the  estimates 
\begin{align}\label{eq4.6.1}
|F_q(y)|_0&\leq C|y|_0\,(|y|_{\alpha}+|v|_0),\\
|F_q(y)|_{\alpha}&\leq C|y|_0\,|y|_{\alpha},\label{eq4.6.2}\\
|F_q(y)-F_q(\bar{y})|_0&\leq C\big(|y-\bar{y}|_0\,(|y|_{\alpha}+|\bar{y}|_{\alpha})+|y-\bar{y}|_0\,|v|_0
  +|\bar{y}|_0\, |v-\bar{v}|_0\big),\label{eq4.6.3}\\
|F_q(y)-F_q(\bar{y})|_{\alpha}
     &\leq |y-\bar{y}|_{\alpha}\,(|y|_0+|\bar{y}|_0) 
   \label{eq4.6.4}
\end{align}
are true, where $C=C(\delta_1, q_0)$ and $|q|\le q_0$.
\end{lemma}
\begin{proof}
Let $|y|_\beta,|\bar{y}|_\beta \le \delta\le \delta_1$. From the proof of Lemma~8.3 
in \cite{GLS2} we recall the 
representation
\begin{equation*}
F_q(y)=I_1(y)+I_2(y)+I_3(y)+I_4(y)+I_5(y),
\end{equation*}
where $Y_q=(U_q,V_q)$, $y=(u,v)$,
\begin{align*}
I_1(y)&=\int_0^1\left( \partial_ur(Y_q+ty)-\partial_ur(Y_q) \right)uV_q\dd t,\\
I_2(y)&=\int_{0}^{1}\left(\partial_ur(Y_q+ty)u\right)tv\dd t,\\
I_3(y)&=\int_0^1\left( \partial_vr(Y_q+ty)-\partial_vr(Y_q) \right)vV_q\dd t,\\
I_4(y)&=\int_0^1\left( \partial_vr(Y_q+ty) v\right)tv\dd t,\\
I_5(y)&=\int_0^1\left( r(Y_q+ty)-r(Y_q) \right)v\dd t,
\end{align*}
and the  function $r\in C^2(\RR^n,\RR^{n\times n})$ is given by 
\begin{equation*}
r(u,v)=\int_0^1\partial_vR(u,tv)\dd t.
\end{equation*}
We note that $r$ is only applied to functions which are uniformly bounded by $C(1+\delta_1)$. It is then 
straightforward to check the inequalities $|I_j(y)|_0\leq C|y|_0\,|v|_0$ for $j\in \{2,\dots,5\}$ and 
$|I_j(y)|_{\alpha}\leq C|y|_0\,|y|_{\alpha}$ for $j\in \{1,2,\dots,5\}$. Since 
$uV_q=(\gamma_{\alpha}u)(\gamma_{\alpha}^{-1}V_q)$ and $(\gamma_{\alpha}^{-1}V_q)\in BUC^1(\mathbb{R})^{n_1}$
 by Lemma~3.7  of \cite{GLS2}, we can further estimate $|I_1(y)|_0\leq C|y|_0\,|y|_{\alpha}$, 
 finishing the proof of \eqref{eq4.6.1} and \eqref{eq4.6.2}. Here and below the constants only depend on 
 $\delta_1$  and $q_0$.

To show \eqref{eq4.6.3} and \eqref{eq4.6.4}, we deal with each integral $I_j$ separately.
The terms $|y-\bar{y}|_0\,(|y|_{\alpha}+ |\bar{y}|_{\alpha})$ and $|y-\bar{y}|_{\alpha}\,(|y|_0+|\bar{y}|_0)$ 
come from $I_1$ while the remaining ones originate from $I_2$ through $I_5$. We first represent
$I_1(y)-I_1(\bar{y})$ as
\begin{equation}\label{eq4.6.5}
\begin{aligned}
I_1(y)-I_1(\bar{y})
 &=\int_0^1\int_0^1\partial_y\partial_ur(Y_q+st(y-\bar{y})+t\bar{y})uV_qt(y-\bar{y})\dd s\dd t\\
&\qquad +\int_0^1\int_0^1\partial_y\partial_ur(Y_q+st\bar{y})(u-\bar{u})V_qt\bar{y}\dd s\dd t.
\end{aligned}
\end{equation}
Using $uV_q(y-\bar{y})=(\gamma_{\alpha}u)(\gamma_{\alpha}^{-1}V_q)(y-\bar{y})$ and 
$(u-\bar{u})V_q\bar{y}=(u-\bar{u})\,\gamma_{\alpha}^{-1}V_q\,\gamma_{\alpha}y$ as above, we conclude that
$|I_1(y)-I_1(\bar{y})|_0\leq C|y-\bar{y}|_0\,(|y|_{\alpha}+|\bar{y}|_{\alpha})$. If we multiply 
\eqref{eq4.6.5} by $\gamma_{\alpha}$, we directly estimate 
$|I_1(y)-I_1(\bar{y})|_{\alpha}\leq C(|y|_0+|\bar{y}|_0)\,|y-\bar{y}|_{\alpha}$ 
since $|u|\leq |y|$.  Likewise, we write $I_5(y)-I_5(\bar{y})$ as 
\begin{equation}\label{eq4.6.6}
\begin{aligned}
I_5(y)-I_5(\bar{y})
&=\int_0^1\int_0^1\left(\partial_yr(Y_q+sty)-\partial_yr(Y_q+st\bar{y})\right)tyv\dd s\dd t\\
&\qquad+\int_0^1\int_0^1\partial_yr(Y_q+ts\bar{y})tv(y-\bar{y})\dd s\dd t\\
&\qquad+\int_0^1\int_0^1\partial_yr(Y_q+st\bar{y})t\bar{y}(v-\bar{v})\dd s\dd t
\end{aligned}
\end{equation}
and obtain the bounds $|I_5(y)-I_5(\bar{y})|_0\leq C(|y-\bar{y}|_0\,|v|_0+|\bar{y}|_0\,|v-\bar{v}|_0)$, 
recalling that $|y|_0\leq \delta_1$ by assumption. After multiplying \eqref{eq4.6.6} by $\gamma_{\alpha}$,
it also follows that
$|I_5(y)-I_5(\bar{y})|_{\alpha}\leq C(|y|_0\,|y-\bar{y}|_{\alpha}+|\bar{y}|_0\,|v-\bar{v}|_{\alpha})$ 
since $|v|\leq |y|$. Similarly, the formulas
\begin{align}\label{eq4.6.7}
I_2(y)-I_2(\bar{y})&=\int_0^1\big(\partial_ur(Y_q+ty)-\partial_ur(Y_q+t\bar{y})\big)utv\dd t\\
&\qquad+\int_{0}^{1}\partial_ur(Y_q+t\bar{y})(u-\bar{u})tv\dd t
         +\int_0^1\partial_ur(Y_q+t\bar{y})\bar{u}t(v-\bar{v})\dd t, \notag\\
I_4(y)-I_4(\bar{y})&=\int_0^1\big(\partial_vr(Y_q+ty)-\partial_vr(Y_q+t\bar{y})\big)vtv\dd t\label{eq4.6.8}\\
&\qquad+\int_{0}^{1}\partial_vr(Y_q+t\bar{y})(v-\bar{v})tv\dd t
   +\int_0^1\partial_vr(Y_q+t\bar{y})\bar{v}t(v-\bar{v})\dd t \notag
\end{align}
imply  the inequalities 
\begin{align*}
|I_2(y)-I_2(\bar{y})|_0 &\leq C(|y-\bar{y}|_0\,|v|_0+|\bar{y}|_0\,|v-\bar{v}|_0),\\
|I_4(y)-I_4(\bar{y})|_0 &\leq C(|y-\bar{y}|_0\,|v|_0+|\bar{y}|_0\,|v-\bar{v}|_0).
\end{align*}
Multiplying \eqref{eq4.6.7} and \eqref{eq4.6.8} by $\gamma_{\alpha}$, we also derive
\begin{align*}
|I_2(y)-I_2(\bar{y})|_{\alpha} &\leq C(|y|_0\,|y-\bar{y}|_{\alpha}+|\bar{y}|_0\,|v-\bar{v}|_{\alpha}),\\
|I_4(y)-I_4(\bar{y})|_{\alpha} &\leq C(|y|_0\,|y-\bar{y}|_{\alpha}+|\bar{y}|_0\,|v-\bar{v}|_{\alpha}).
\end{align*}
We finally compute
\begin{equation*}
\begin{aligned}
I_3(y)-I_3(\bar{y})
&=\int_0^1\int_0^1\partial_y\partial_vr(Y_q+st(y-\bar{y})+t\bar{y})vV_qt(y-\bar{y})\dd s\dd t\\
&\qquad +\int_0^1\int_0^1\partial_y\partial_vr(Y_q+st\bar{y})(v-\bar{v})V_qt\bar{y}\dd s\dd t.
\end{aligned}
\end{equation*}
Again we infer that
\begin{align*}
|I_3(y)-I_3(\bar{y})|_0&\leq C(|y-\bar{y}|_0\,|v|_0+|\bar{y}|_0\,|v-\bar{v}|_0),\\
|I_3(y)-I_3(\bar{y})|_{\alpha}&\leq C(|y|_0 \,|y-\bar{y}|_{\alpha}+|\bar{y}|_0\,|v-\bar{v}|_{\alpha}).
\end{align*} 
This completes the proof of the lemma.
\end{proof}

\begin{remark}\label{rem:e-beta}
It follows from the observations after Remark~\ref{RB} that the realization of $L_q$ in 
$\cE_\beta=\cE_0\cap\cE_\alpha$ generates a strongly continuous semigroup. The Lipschitz properties 
proved in the above lemma thus imply that the semilinear equation \eqref{eq4.3.3} is locally wellposed 
also in $\cE_\beta$, cf.\ Remark~\ref{DL}. \hfill $\Diamond$
\end{remark}

We next  establish basic properties of the Lyapunov-Perron operator $\Phi_q(y,z_0)$ defined by
\begin{equation}\label{eq4.7.1}
\Phi_q(y,z_0)(t)=T_q(t)P_q^{\mathrm{s}}z_0 +\int_0^t T_q(t-\tau)P_q^{\mathrm{s}}F_q(y(\tau))\dd\tau
 -\int_t^{\infty}\!\!P_q^{\mathrm{c}}F_q(y(\tau))\dd\tau,
\end{equation}
where $|q|\le q_0$ and $z_0\in\mathcal{E}_0\cap\mathcal{E}_{\alpha}=\cE_\beta$ satisfies
\begin{equation}\label{eq4.7.2}
|z_0|_{\beta}=\max\{ |z_0|_0,|z_0|_{\alpha} \}\leq \delta_0,
\end{equation}
for some  $\delta_0>0$. Here we use that  $P_q^c$ maps into the kernel of the 
generator of $\{T_q(t)\}_{t\ge0}$, see Remark~\ref{Projq}, so that the semigroup is just the identity 
on the range of $P_q^c$ and we can omit it in  the 
second integral in \eqref{eq4.7.1}.

For a continuous map $y=(u,v):\mathbb{R}\rightarrow \mathcal{E}_\beta=\cE_0\cap\cE_\alpha$ 
we define the norms
\begin{equation*}
\|y\|_{\omega,\alpha}=\sup\limits_{t\geq 0}\e^{\omega t}|y(t)|_{\alpha},\quad
\|y\|_{0,0}=\sup\limits_{t\geq 0}|y(t)|_0,\quad
\|v\|_{\omega,0}=\sup\limits_{t\geq 0} \e^{\omega t}|v(t)|_0,
\end{equation*}
where $\omega>0$ is specified below and $\alpha=(\alpha_-,\alpha_+)$ is given by \eqref{alphaomega}.
Let $\delta>0$. Then  $(\mathbb{B}_{\delta},\|\cdot\|)$ is the set of continuous functions 
$y:\mathbb{R}\rightarrow \mathcal{E}_0\cap \mathcal{E}_{\alpha}$ such that 
\begin{equation}\label{eq4.7.3}
\|y\|:=\max\left( \|y\|_{\omega,\alpha},\|y\|_{0,0},\|v\|_{\omega,0}\right)\leq \delta.
\end{equation}

We recall from Hypothesis~\ref{hypos} and \eqref{exp-stab} the exponential estimates
\begin{equation}\label{eq4.7.5}
\|S_2(t)\|_{\mathcal{B}(\mathcal{E}_0)}\leq C\e^{-\rho t},\qquad
\|T_q(t)P_q^{\mathrm{s}}\|_{\mathcal{B}(\mathcal{E}_{\alpha})}\leq C\e^{-\nu t}
\end{equation}
for $t\ge0$. For technical reasons (see the next proof), if necessary we have to modify these exponents 
 such that 
\begin{equation}\label{eq4.7.4}
0<\omega<\rho<\nu.
\end{equation}
 This is always possible, though one may lose information here. By Lemma~\ref{lemBT}, 
the semigroup $\{ T_q(t) \}_{t\geq 0}$ is bounded in  $\mathcal{E}_{\beta}$.
The above constants do not depend on $|q|\le q_0$. 

\begin{lemma}\label{Lproper}
Take $q_0>0$.
Let $\delta>0$ and $\delta_0=\delta_0(\delta)>0$ be small enough. For each 
$z_0\in\mathbb{B}_{\delta_0}(\left|\cdot\right|_{\beta})$  the Lyapunov-Perron 
operator $y\mapsto\Phi_q(y,z_0)$ leaves $\mathbb{B}_{\delta}(\|\cdot\|)$ invariant and is a strict 
contraction on this ball for  all $\left|q\right|\leq q_0$. Moreover, for the norm $\|\cdot\|$ defined
in \eqref{eq4.7.3} one has 
\begin{equation}\label{LipPhiz}
\|\Phi_q(y,z_0)-\Phi_q(\bar{y},\bar{z}_0)\|\leq C|z_0-\bar{z}_0|_{\beta}+C\delta\|y-\bar{y}\|
\end{equation}
for some $C>0$ and all $z_0,\bar{z}_0\in \mathbb{B}_{\delta_0}(|\cdot|_{\beta})$, 
$y,\bar{y}\in \mathbb{B}_{\delta}(\|\cdot\|)$, and $|q|\leq q_0$.
\end{lemma}

\begin{proof}
Let $t\ge0$, $\delta,\delta_0>0$, $z_0,\bar{z}_0\in \mathbb{B}_{\delta_0}(|\cdot|_{\beta})$, 
$y,\bar{y}\in \mathbb{B}_{\delta}(\|\cdot\|)$, and $|q|\leq q_0$. Below the constants are uniform 
for  $\delta$, $\delta_0$ and $q$ in bounded subsets.
By $\pi_1y=u$ and $\pi_2y=v$, we denote the projection of $y=(u,v)$ onto its first and second components.
Formulas \eqref{eq4.5.1} and \eqref{eq4.5.2} yield
\begin{equation}\label{eq4.8.1}
\pi_2T_q(t-\tau)=S_2(t-\tau)\pi_2+\int_{\tau}^{t}S_2(t-s)\pi_2B_qT_q(s-\tau)\dd s, \quad  0\le \tau\le t.
\end{equation}

1a) Using \eqref{eq4.8.1}, \eqref{eq4.7.5}, and Remarks~\ref{Projq} and \ref{RB}, the second component of 
the first integral in \eqref{eq4.7.1} can be estimated by
\begin{align}\label{eq4.8.2}
 \e^{\omega t}&\left|\pi_2\int_0^tT_q(t-\tau)P_q^{\mathrm{s}}F_q(y(\tau))\dd\tau\right|_0\\
&\leq C\e^{\omega t}\int_0^t\left( \e^{-\rho(t-\tau)}|F_q(y(\tau))|_\beta
  +\int_{\tau}^t\e^{-\rho(t-s)}\e^{-\nu(s-t)}|F_q(y(\tau))|_{\alpha}\dd s \right)\D\tau,  \notag
\end{align}
since
\begin{align}
&|S_2(t-\tau)\pi_2P_q^sF_q(y(\tau))|_0\le
\|S_2(t-\tau)\pi_2\|_{\mathcal{B}(\mathcal{E}_0)}\|P_q^s\|_{\mathcal{B}(\mathcal{E}_\beta,\mathcal{E}_0)}|F_q(y(\tau))|_\beta,\label{NES1}\\
&|S_2(t-s)\pi_2B_qT_q(s-\tau)P_q^sF_q(y(\tau))|_0\label{NES2}\\
&\qquad\qquad\le
\|S_2(t-s)\pi_2\|_{\mathcal{B}(\mathcal{E}_0)}
\|B_q\|_{\mathcal{B}(\mathcal{E}_\alpha,\mathcal{E}_0)}\|T_q(s-\tau)P_q^s\|_{\mathcal{B}(\mathcal{E}_\alpha)}|F_q(y(\tau))|_\alpha.\notag
\end{align}
Because of  \eqref{eq4.6.1} and \eqref{eq4.6.2}, the formulas \eqref{eq4.8.2} and \eqref{eq4.7.4} yield
\begin{align*}
\e^{\omega t}&\left| \pi_2\int_0^tT_q(t-\tau)P_q^{\mathrm{s}}F_q(y(\tau))\dd\tau \right|_0\\
&\leq  C\e^{\omega t}\int_0^t\bigg( \e^{-\rho (t-\tau)}\e^{-\omega \tau}\e^{\omega \tau}
             (|y(\tau)|_{\alpha}+|v(\tau)|_0)\,|y(\tau)|_0\\
& \qquad \qquad+\int_{\tau}^t\e^{-\rho(t-s)}\e^{-\nu(s-t)}\e^{-\omega \tau}\e^{\omega\tau}
    |y(\tau)|_{\alpha}\,|y(\tau)|_0\dd s \bigg)\D\tau\\
& \leq  C(\|y\|_{\omega,\alpha}+\|v\|_{\omega,0})\,\|y\|_{0,0}\int_0^t \e^{(\omega-\rho)(t-\tau)}\dd\tau\\
& \qquad  +C\|y\|_{\omega,\alpha}\,\|y\|_{0,0}\int_0^t\e^{\omega(t-\tau)}\left( \int_{\tau}^t\e^{-\rho(t-s)}
        \e^{-\nu(s-\tau)} \dd s\right)\D\tau\\
 &\leq  C\|y\|^2\leq  C\delta^2.
\end{align*}
We next employ  \eqref{eq4.7.5}, \eqref{eq4.6.2} and \eqref{eq4.7.4} to bound 
\begin{align*}
\e^{\omega t}\left| \int_0^tT_q(t-\tau)P_q^{\mathrm{s}}F_q(y(\tau))\dd\tau \right|_{\alpha}
& \leq C\int_0^t\e^{\omega t}\e^{-\nu(t-\tau)}\e^{-\omega \tau}
     \e^{\omega \tau}|y(\tau)|_0\,|y(\tau)|_{\alpha}\dd\tau \\
&\leq C\|y\|_{0,0}\,\|y\|_{\omega,\alpha} \leq C\delta^2.
\end{align*}
To finish with the first integral in \eqref{eq4.7.1}, it remains to control the $\left|\cdot\right|_0$ norm 
of its first component. Here \eqref{eq4.5.1}, \eqref{eq4.5.4}, Remark~\ref{Projq} (in particular, that $P_q^s\in\mathcal{B}(\mathcal{E}_\beta,\mathcal{E}_0)$), Remark~\ref{RB}, 
\eqref{eq4.6.1}, \eqref{eq4.7.5}, \eqref{eq4.6.2} and \eqref{eq4.7.4}  imply the  inequalities
\begin{align*}
\bigg|& \pi_1\int_0^tT_q(t-\tau)P_q^{\mathrm{s}}F_q(y(\tau))\dd\tau \bigg|_0\\
&= \bigg| \pi_1\!\!\int_0^t\! S(t-\tau)P_q^{\mathrm{s}}F_q(y(\tau))\dd\tau 
 + \pi_1\!\!\int_0^t\int^t_\tau \! S(t-s)B_qT_q(s-\tau)P_q^{\mathrm{s}}F_q(y(\tau))\dd s\dd\tau \bigg|_0\\
&\leq  C\int_0^t|F_q(y(\tau))|_\beta\dd\tau + C\|B_q\|_{\mathcal{B}(\mathcal{E}_{\alpha},\mathcal{E}_0)}
     \int_0^t\int^t_{\tau}|T_q(s-\tau)P_q^{\mathrm{s}}F_q(y(\tau))|_{\alpha}\dd s\dd\tau\\
&\leq C\!\!\int_0^t\!\e^{-\omega\tau}|y(\tau)|_0\,\e^{\omega\tau} (|y(\tau)|_{\alpha}+|v(\tau)|_0)\dd\tau 
     + C\!\!\int_0^t\!\int^t_{\tau}\!\e^{-\nu (s-\tau)}|y(\tau)|_0\,|y(\tau)|_{\alpha}\dd s\dd\tau\\
&\leq C\|y\|_{0,0}\,(\|y\|_{\omega,\alpha}+\|v\|_{\omega,0})\!\int_0^t \!\e^{-\omega \tau}\D\tau
  +C\|y\|_{0,0}\|y\|_{\omega,\alpha}\!\int_0^t\!\int^t_\tau\!\e^{-\nu (s-\tau)}\D s\,\e^{-\omega \tau}\dd\tau\\
&\leq  C\delta^2.
\end{align*}

1b) We now treat the term $T_q(t)P_q^{\mathrm{s}}z_0$ in \eqref{eq4.7.1}. From \eqref{eq4.7.5} 
and \eqref{eq4.7.4} we infer
\begin{equation*}
\e^{\omega t}|T_q(t)P_q^{\mathrm{s}}z_0|_{\alpha}\leq C\e^{\omega t}\e^{-\nu t}|z_0|_{\alpha}\leq C\delta_0.
\end{equation*}
By means of  \eqref{eq4.5.1}, \eqref{eq4.5.4}, Remark~\ref{Projq} (in particular, that 
$P_q^s\in\mathcal{B}(\mathcal{E}_\beta,\mathcal{E}_0)$) and Remark~\ref{RB}, as well as \eqref{eq4.7.5}, we next compute 
\begin{align*}
|\pi_1T_q(t)P_q^{\mathrm{s}}z_0|_0
&\leq |S(t)P_q^{\mathrm{s}}z_0|_0+\int_0^t|S(t-s)B_qT_q(s)P_q^{\mathrm{s}}z_0|_0\dd s\\
&\leq C|z_0|_\beta+C\int_0^t\|B_q\|_{\mathcal{B}(\mathcal{E}_{\alpha},\mathcal{E}_0)}
 \e^{-\nu s}|z_0|_{\alpha}\dd s \ \ 
\leq C\delta_0.
\end{align*}
Finally, formulas \eqref{eq4.8.1}, \eqref{eq4.7.5}, Remarks~\ref{Projq}  and \ref{RB}, as well as 
inequality \eqref{eq4.7.4} imply
\begin{align*}
\e^{\omega t}|\pi_2T_q(t)P_q^{\mathrm{s}}z_0|_0 
 &\leq  \e^{\omega t} |S_2(t)\pi_2P_q^{\mathrm{s}}z_0|_0
     +\int_0^t|S_2(t-s)\pi_2B_qT_q(s)P_q^{\mathrm{s}}z_0|_0\dd s\\
&\leq  C\e^{(\omega -\rho)t}|z_0|_\beta+C\int_0^t\e^{\omega t}\e^{-\rho(t-s)}\e^{-\nu s}|z_0|_{\alpha}\dd s\\
&\leq  C\delta_0.
\end{align*}

1c) To show the invariance, it remains to bound the norms of the last integral in \eqref{eq4.7.1}. 
Remark~\ref{Projq} (in particular, that $P_q^c\in\mathcal{B}(\mathcal{E}_\alpha,\mathcal{E}_\beta)$) and  estimate  \eqref{eq4.6.2} yield
\begin{align*}
\e^{\omega t}\left|\int_t^{\infty}P_q^{\mathrm{c}}F_q(y(\tau))\dd\tau\right|_{\beta}
&\leq C\int_t^{\infty}\e^{\omega t}|F_q(y(\tau))|_{\alpha}\dd\tau\\
&\leq C\int_t^{\infty}\e^{\omega t}\e^{-\omega \tau}\e^{\omega \tau}|y(\tau)|_0\,|y(\tau)|_{\alpha}\dd\tau\\
& \leq C\|y\|_{0,0}\,\|y\|_{\omega,\alpha} \leq C\delta^2.
\end{align*}
We thus have shown that $\Phi_q(\cdot,z_0)$ leaves the ball $\mathbb{B}_\delta(\|\cdot\|)$ invariant 
if first $\delta>0$ and then $\delta_0>0$ are chosen small enough.

2) For the contractivity we have to estimate the difference
\begin{align}\label{eq4.8.3}
\Phi_q(y,z_0)-\Phi_q(\bar{y},z_0)
&=\int_0^t T_q(t-\tau)P_q^{\mathrm{s}}\big( F_q(y(\tau))-F_q(\bar{y}(\tau)) \big)\dd\tau\\
& \qquad - \int_t^{\infty}P_q^{\mathrm{c}} \big(F_q(y(\tau))-F_q(\bar{y}(\tau)) \big)\dd\tau.\notag
\end{align}
2a) Using \eqref{eq4.5.1}, \eqref{eq4.5.4}, \eqref{eq4.7.5}, Remark~\ref{Projq} (in particular, that $P_q^s\in\mathcal{B}(\mathcal{E}_\beta,\mathcal{E}_0)$) and Remark~\ref{RB}, we 
bound the first integral by
\begin{align*} 
\Big|\int_0^t &T_q(t-\tau)P_q^{\mathrm{s}}\left( F_q(y(\tau))-F_q(\bar{y}(\tau))\right)\D\tau \Big|_0\notag\\
&\leq \left| \int_0^t S(t-\tau)P_q^{\mathrm{s}}\big( F_q(y(\tau))-F_q(\bar{y}(\tau)) \big) \dd\tau \right|_0
    \notag\\
&\qquad +\left|  \int_0^t\int_{\tau}^{t}S(t-s)B_qT_q(s-\tau)P_q^{\mathrm{s}}
    \big(   F_q(y(\tau))-F_q(\bar{y}(\tau))  \big) \dd s\dd\tau\right|_0\notag \\
& \leq C\int_0^t\left|  F_q(y(\tau))-F_q(\bar{y}(\tau))  \right|_\beta\dd\tau \notag\\
& \qquad +C \int_0^t\int_{\tau}^{t}\|B_q\|_{\mathcal{B}(\mathcal{E}_0,\mathcal{E}_{\alpha})} 
    \e^{-\nu(s-\tau)}\left|  F_q(y(\tau))-F_q(\bar{y}(\tau)) \right|_{\alpha} \dd s \dd\tau.\notag
\end{align*} 
The inequalities  \eqref{eq4.6.3} and \eqref{eq4.6.4} then lead to
\begin{align*}
\Big|& \int_0^t T_q(t-\tau)P_q^{\mathrm{s}}\left( F_q(y(\tau))-F_q(\bar{y}(\tau)) \right)\dd\tau \Big|_0\\
& \leq  C\int_0^t\e^{-\omega \tau} \Big[ |y(\tau)-\bar{y}(\tau)|_0\,\e^{\omega \tau} 
   (|y(\tau)|_{\alpha}+|\bar{y}(\tau)|_{\alpha}+|v(\tau)|_0) \\
 &\qquad\qquad   + |\bar{y}(\tau)|_0\,\e^{\omega \tau}|v(\tau)-\bar{v}(\tau)|_0  
 + \e^{\omega \tau}\, |y(\tau)-\bar{y}(\tau)|_\alpha\, (|y(\tau)|_0+|\bar{y}(\tau)|_0)\Big]\dd\tau \\
&\quad +C\int_0^t\e^{-\omega \tau}\e^{\omega \tau}|y(\tau)-\bar{y}(\tau)|_{\alpha}\,
    (|y(\tau)|_0 +|\bar{y}(\tau)|_0 ) \dd\tau\\
&\leq C\|y-\bar{y}\|_{0,0}\,\big( \|y\|_{\omega,\alpha}+  \|\bar{y}\|_{\omega,\alpha}+\|v\|_{\omega,0} \big)  
         +C\|\bar{y}\|_{0,0}\,\|v-\bar{v}\|_{\omega,0}\\
&\qquad +C \|y-\bar{y}\|_{\omega,\alpha}\,(\|y\|_{0,0}+ \|\bar{y}\|_{0,0})   \\
&\leq  C\delta\,\|y-\bar{y}\|.
\end{align*}
The $\left|\cdot\right|_{\alpha}$-norm of the first integral in \eqref{eq4.8.3} is estimated by
\begin{align*}
\e^{\omega t} &\left| \int_0^tT_q(t-\tau)P_q^{\mathrm{s}}(F_q(y(\tau)-F_q(\bar{y}(\tau)))
         \dd\tau \right|_{\alpha}\\
&\leq  C\int_0^t \e^{\omega t}\e^{-\nu (t-\tau)}\e^{-\omega \tau}\e^{\omega \tau}   
   |y(\tau)-\bar{y}(\tau)|_{\alpha} \,( |y(\tau)|_0 + |\bar{y}(\tau)|_0) \dd\tau \\
&\leq  C\int_0^t \e^{(\omega-\nu)(t-\tau)} \dd\tau \, \|y-\bar{y}\|_{\omega,\alpha}\,(\|y\|_{0,0}+\|\bar{y}\|_{0,0})\\
&\leq  C\delta \,\|y-\bar{y}\|,
\end{align*}
employing \eqref{eq4.7.5}, \eqref{eq4.6.4}, and \eqref{eq4.7.4}. As in \eqref{NES1} and \eqref{NES2}, for the second 
component we use
formulas \eqref{eq4.8.1} and \eqref{eq4.7.5}, Remark~\ref{Projq} (in particular, that $P_q^s\in\mathcal{B}(\mathcal{E}_\beta,\mathcal{E}_0)$) and Remark~\ref{RB}, 
as well as inequalities \eqref{eq4.6.3}, \eqref{eq4.6.4} and \eqref{eq4.7.4} to derive the estimates
\begin{align*}
&\e^{\omega t} \left|\pi_2 \int_0^tT_q(t-\tau)P_q^{\mathrm{s}}(F_q(y(\tau)-F_q(\bar{y}(\tau)))\dd\tau  
         \right|_{0}\\
&\leq  \e^{\omega t}\left| \int_0^tS_2(t-\tau)\pi_2P_q^{\mathrm{s}}(F_q(y(\tau)-F_q(\bar{y}(\tau)))\dd\tau   
     \right|_0\\
& \qquad +\e^{\omega t}\left| \int_0^t\int_{\tau}^{t}S_2(t-s)\pi_2B_qT_q(s-\tau)P_q^{\mathrm{s}}
           (F_q(y(\tau)-F_q(\bar{y}(\tau)))\dd s\dd\tau \right|_0\\
&\leq  C\int_0^t\e^{\omega (t-\tau)} \e^{-\rho (t-\tau)} \e^{\omega \tau} 
\Big[|y(\tau)-\bar{y}(\tau)|_{0}\,( |y(\tau)|_{\alpha}+ |\bar{y}(\tau)|_{\alpha}+|v(\tau)|_0) \\
 &\qquad  +|\bar{y}(\tau)|_0\,|v(\tau)-\bar{v}(\tau)|_0 
  + |y(\tau)-\bar{y}(\tau)|_{\alpha}\,(|y(\tau)|_0  + |\bar{y}(\tau)|_0)\Big]\dd\tau \\
&\ \   +\! C\!\!\int_0^t\!\!\int_{\tau}^{t}\!\!\e^{\omega t-\rho(t-s)} \|B_q\|_{\mathcal{B}(\cE_{\alpha},\cE_0)} 
 \e^{-\nu(s-\tau)-\omega \tau}\e^{\omega \tau} |y(\tau)-\bar{y}(\tau)|_{\alpha}(|y(\tau)|_0  + |\bar{y}(\tau)|_0)\D\tau\\
&\leq  C \Big( \|y-\bar{y}\|_{0,0}( \|y\|_{\omega,\alpha} + \|\bar{y}\|_{\omega,\alpha}+\|v\|_{\omega,0}) 
  + \|\bar{y}\|_{0,0}\,\|v-\bar{v}\|_{\omega,0} \\
&\qquad   +\|y-\bar{y}\|_{\omega,\alpha}\,(\|y\|_{0,0}+\|\bar{y}\|_{0,0}) \Big)\\
&\leq  C\delta \,\|y-\bar{y}\|.
\end{align*}
As a result, the $\|\cdot\|$-norm of the first integral in \eqref{eq4.8.3} is dominated by 
$C\delta\,\|y-\bar{y}\|$. 

2b) For the second integral in \eqref{eq4.8.3}, Remark~\ref{Projq} and  inequality  
 \eqref{eq4.6.4} yield
\begin{align*}
\e^{\omega t}&\Big|\int_t^{\infty} P_q^{\mathrm{c}}(F_q(y(\tau))-F_q(\bar{y}(\tau)))   
        \dd\tau\Big|_{\beta}\\
&\leq  C\int_t^{\infty}\e^{\omega(t-\tau)}\e^{\omega \tau} |y(\tau)-\bar{y}(\tau)|_{\alpha}\,  
  (|y(\tau)|_0+ |\bar{y}(\tau)|_0)\dd\tau \\ 
&\leq  C\|y-\bar{y}\|_{\omega,\alpha}(\|y\|_{0,0}+\|\bar{y}\|_{0,0})\\
&\leq  C\delta\,\|y-\bar{y}\|.
\end{align*}
 We have thus established 
\begin{equation}\label{est:phi-lip}
\|\Phi_q(y,z_0)-\Phi_q(\bar{y},z_0)\|\leq C\delta\,\|y-\bar{y}\|,
\end{equation}
finishing the proof of the first part of Lemma~\ref{Lproper}. 

3) It remains to show  the estimate
\begin{equation}  \label{LP2}
\|\Phi_q(y,z_0)-\Phi_q(y,\bar{z}_0)\|=\| T_q(\cdot)P_q^s(z_0-\bar{z}_0) \|\leq C|z-\bar{z}_0|_{\beta}.
\end{equation}
The inequalities \eqref{eq4.7.5} and \eqref{eq4.7.4} first imply that 
\[\|T_q(\cdot)P_q^s(z_0-\bar{z}_0)\|_{\omega,\alpha}\leq C|z-z_0|_{\alpha}.\]
 To treat the norm $\|\cdot\|_{0,0}$, from \eqref{eq4.5.1}, \eqref{eq4.5.4}, Remark~\ref{Projq} (in particular, that $P_q^s\in\mathcal{B}(\mathcal{E}_\beta,\mathcal{E}_0)$), Remark~\ref{RB} and \eqref{eq4.7.5}, we conclude
\begin{align*}
|T_q(t)P_q^s(z_0-\bar{z}_0)|_0
&\leq |S(t)P_q^s(z_0-\bar{z}_0)|_0 + \int_0^t|S(t-s)B_qT_q(s)P_q^s(z_0-\bar{z}_0)|_0\dd s\\
&\leq C|z_0-\bar{z}_0|_\beta+C\int_0^t\e^{-\nu s}\dd s\;|z_0-\bar{z}_0|_{\alpha}\\
&\le C|z_0-\bar{z}_0|_{\beta}.
\end{align*}
In a similar way, formula  \eqref{eq4.8.1}, Remarks~\ref{Projq} and \ref{RB}, as well as inequalities 
\eqref{eq4.7.5}  and \eqref{eq4.7.4} yield the remaining bound for  
$\|\pi_2 T_q(\cdot)P_q^s(z_0-\bar{z}_0)\|_{\omega,0}$ so that \eqref{LP2} follows.
\end{proof}

\section{Stable manifolds}\label{sec3}
For a small $q_0>0$ and each $q\in[- q_0,q_0]$, we now construct a function 
$\phi_q:\Ran(P_q^s)\rightarrow P_q^c$ whose graph contains $Y_q$ and it is a stable manifold $\mathcal{M}^s_q$ 
for the system \eqref{eq4.2.3}. We further prove that the sets $\mathcal{M}^s_q$ satisfy
the standard properties of  stable manifolds and that they foliate a small neighborbood of $Y_0$. 

Let $\delta,\delta_0>0$  be sufficiently small and $q_0>0$. Take $|q|\leq q_0$ and 
$z_0\in \Ran(P_q^s)\cap \mathbb{B}_{\delta_0}(|\cdot|_{\beta})$. Lemma~\ref{Lproper} then yields
a unique function $y_{z_0}^q:\mathbb{R}_+\rightarrow \mathcal{E}_\beta$   which belongs to 
$\mathbb{B}_{\delta}(\|\cdot\|)$ and is a fixed point of the Lyapunov-Perron operator $\Phi_q(\cdot,z_0)$; 
that is,
\begin{equation}
\label{LPFP}
y_{z_0}^q(t)=T_q(t)z_0+\int_0^t T_q(t-\tau)P_q^sF_q(y_{z_0}^q(\tau))\dd\tau
     -\int_t^{\infty}P_q^cF_q(y_{z_0}^q(\tau))\dd\tau
\end{equation}
for $t\geq 0$. At $t=0$ we obtain the identity
\begin{equation*}
y_{z_0}^q(0)=z_0-\int_0^{\infty}P_q^cF_q(y_{z_0}^q(\tau))\dd\tau
\end{equation*}
for all $z_0\in\Ran(P_q^s)\cap \mathbb{B}_{\delta_0}(|\cdot|_{\beta})$. We define the function 
$\phi_q:\Ran (P_q^s)\cap\mathbb{B}_{\delta_0}(|\cdot|_{\beta})\rightarrow \Ran(P_q^c)$ by
\begin{equation}
\label{defnphi}
\phi_q(z_0)=-\int_0^{\infty}P_q^cF_q(y_{z_0}^q(\tau))\dd\tau.
\end{equation}
In this notation, we have $y_{z_0}^q(0)=z_0+\phi_q(z_0)$ so that $y_{z_0}^q(0)$ 
belongs to  the graph $\text{\rm{graph}}_{\delta_0}\phi_q$ of $\phi_q$ over the small 
neighborhood  $\Ran(P_q^s) \cap \mathbb{B}_{\delta_0}(|\cdot|_{\beta})$ of 0. Adding and substracting  
the term $\int_0^t P_q^cF_q(y_{z_0}^q(\tau))\dd\tau$, we deduce from \eqref{LPFP} that the 
fixed point $y=y_{z_0}^q$ of the Lyapunov-Perron operator satisfies the equation
\begin{equation}
\label{VOCF}
y(t)=T_q(t)y(0)+\int_0^t T_q(t-\tau)F_q(y(\tau))\dd\tau,\qquad t\geq 0.
\end{equation}
Consequently,  $y=y_{z_0}^q$ is the mild solution of the nonlinear equation \eqref{eq4.3.3} in 
$\mathbb{B}_{\delta}(\|\cdot\|)$, and the function $Y_q+y$ solves \eqref{eq4.2.3} in the mild sense, 
cf.\ Remark~\ref{DL}. By uniqueness,  $y_0^q$ is the 0 function.
 Let also $\bar{z}_0$ belong to $\Ran(P_q^s)\cap\mathbb{B}_{\delta_0}(|\cdot|_{\beta})$. 
Taking a sufficiently small $\delta>0$ in \eqref{LipPhiz}, we deduce the estimates
\begin{equation}\label{y-lip}
\|y_{z_0}^q- y_{\bar{z}_0}^q\|\le C |z_0-\bar{z}_0|_\beta, \qquad \|y_{z_0}^q\|\le C|z_0|_\beta.
\end{equation}
For a number $\eta>0$ to be fixed below, the stable manifold $\mathcal{M}_q^s$ is then defined by
\begin{equation}
\label{DFNMQ}
\mathcal{M}_q^s=\{ Y_q+z_0+\phi_q(z_0):z_0\in \Ran(P_q^s)\cap \mathbb{B}_{\delta_0}(|\cdot|_{\beta}) \}
  \cap (Y_0 + \mathbb{B}_{\eta}(|\cdot|_{\beta})),
\end{equation}
where $|q|\le q_0$ and  $Y_0+\mathbb{B}_\eta(|\cdot|_{\beta})$ is the closed ball  in 
$\cE_\beta=\mathcal{E}_{\alpha}\cap\cE_0$ with radius $\eta$ and centered at the original 
traveling wave  $Y_0$.

\begin{theorem}
\label{STFOLTH}
Assume Hypotheses~\ref{HypSpL} and \ref{hypos}. Let $q_0, \delta, \delta_0,\eta>0$ be sufficiently small,
 $|q|\le q_0$, and $\omega>0$ be given by \eqref{eq4.7.4}. Then the ball 
 $Y_0+\mathbb{B}_{\eta}(|\cdot|_{\beta})$  is foliated by the stable manifolds $\mathcal{M}_q^s$ from
\eqref{DFNMQ} for the nonlinear equation \eqref{eq4.2.3} and the following assertions hold.
\begin{itemize}
\item[(i)] Each $\mathcal{M}_q^s$ is a Lipschitz manifold in $\cE_\beta$. If $Y(0)\in \mathcal{M}_q^s$
and the mild solution $Y(t; Y(0))$ of \eqref{eq4.2.3} belongs to $Y_0+\mathbb{B}_{\eta}(|\cdot|_{\beta})$
for some $t\ge0$, then $Y(t; Y(0))$ is contained in $\mathcal{M}_q^s$.
\item[(ii)] For each $Y(0)\in\mathcal{M}_q^s$ there exists a solution $Y(t;Y(0))$ of \eqref{eq4.2.3} 
 which exists for all $t\geq 0$ and satisfies $|Y(t;Y(0))-Y_q|_{\beta}\leq \delta$ as well as
 \begin{itemize}
   \item[(a)]  $|Y(t;Y(0))-Y_q|_{\alpha}\leq C\e^{-\omega t}\,|Y(0)-Y_q|_{\beta}$,
  \item[(b)] $|\pi_1(Y(t;Y(0))-Y_q)-U_q|_0\leq C\,|Y(0)-Y_q|_{\beta},$
  \item[(c)] $|\pi_2(Y(t;Y(0))-Y_q)-V_q|_0\leq C\e^{-\omega t}\,|Y(0)-Y_q|_{\beta}$
 \end{itemize}
 for all $t\ge0$. Here, $Y_q=(U_q, V_q)= Y_0(\cdot-q)$ is the shifted traveling wave,
 $\pi_1:Y=(U,V)\rightarrow U$, and $\pi_2:Y=(U,V)\rightarrow V$.
\item[(iii)] If $Y(t;Y(0))$, $t\geq 0$, is a mild solution of \eqref{eq4.2.3} with 
$Y(0)\in Y_0+\mathbb{B}_{\eta}(|\cdot|_{\beta})$ that satisfies properties 
 (a)--(c) in item (ii), then $Y(0)$ belongs to $\mathcal{M}_q^s$.
\item[(iv)] For $q\neq \bar{q}$, we have $\mathcal{M}_q^s\cap\mathcal{M}_{\bar{q}}^s=\emptyset$. Moreover,    
  $Y_0+\mathbb{B}_{\eta}(|\cdot|_{\beta})=\bigcup_{|q|\leq q_0}\mathcal{M}_q^s$.
  \item[(v)] The map $[-q_0,q_0]\to \Ran(P_q^c)$; $q\mapsto \phi_q(P_q^sz_0)$, is Lipschitz for each 
     $z_0\in\mathbb{B}_{\delta_0}(|\cdot|_{\beta})$.
\end{itemize}
As a result, for each $Y(0)\in Y_0+\mathbb{B}_{\eta}(|\cdot|_{\beta})$ there exists exactly one shift
$q\in[-q_0,q_0]$ such that $Y(0)\in\mathcal{M}_q^s$.
\end{theorem}

The following lemma will be used in the proof of Theorem~\ref{STFOLTH}. Recall
the definition of the ball $\mathbb{B}_{\delta}(\|\cdot\|)$ in \eqref{eq4.7.3}.
\begin{lemma}
\label{LEQ}
Assume Hypotheses~\ref{HypSpL} and \ref{hypos}. Let $\delta,\delta_0>0$ be chosen small enough, $q_0>0$, 
and let $|q|\leq q_0$. Take $y_0\in \cE_\beta=\mathcal{E}_{\alpha}\cap\mathcal{E}_0$. 
Let $y=Y(\cdot;y_0)\in C([0,t_0), \mathcal{E}_0\cap\mathcal{E}_{\alpha})$ be the mild solution of the 
nonlinear equation \eqref{eq4.3.3} with the initial value $y(0)=y_0$, where $t_0\in(0,\infty]$. Set 
$z_0=P_q^sy_0$ and assume that  $|z_0|_\beta\le \delta_0$. Then the following assertions are equivalent.
\begin{itemize}
\item[(a)] $y_0=z_0+\phi_q(z_0)\in\text{\rm{graph}}_{\delta_0}\phi_q$.
\item[(b)] $y$ can be extended to a global mild solution of \eqref{eq4.3.3} in
$\mathbb{B}_{\delta}(\|\cdot\|)$,  and it is the fixed point $y_{z_0}^q$ 
         of the Lyapunov-Perron operator $\Phi_q(\cdot,z_0)$ from \eqref{eq4.7.1}.
\item[(c)] $y$ can be extended to a global mild solution of \eqref{eq4.3.3} in
$\mathbb{B}_{\delta}(\|\cdot\|)$.
\end{itemize}
\end{lemma}
\begin{proof}
(a)$\Rightarrow$(b): Assertion (a) and the equations  \eqref{defnphi} and \eqref{LPFP} yield
\begin{equation*}
y_0=z_0+\phi_q(z_0)=z_0-\int_0^{\infty}P_q^cF_q(y_{z_0}^q(\tau))\dd\tau=y_{z_0}^q(0),
\end{equation*}
where $y_{z_0}^q\in \mathbb{B}_{\delta}(\|\cdot\|)$ is the fixed point of  $\Phi_q(\cdot,z_0)$.
 Since their initial values  are the same, the mild solutions 
 $y$ and $y_{z_0}^q$ coincide by uniqueness of  \eqref{VOCF}; i.e., (b) holds.

(b)$\Rightarrow$(c): This implication is obvious.  

(c)$\Rightarrow$(a): In view of (c), Lemma~\ref{Lproper} shows that the integral
\begin{equation*}
z_c:=P_q^cy_0+\int_0^{\infty}P_q^cF_q(y(\tau))\dd\tau\in \Ran(P_q^c).
\end{equation*}
exists. Since $y$ solves \eqref{VOCF} and $T_q(t-\tau)$ is the identity on $\Ran (P_q^c)$, we can write
\begin{align*}
y(t)&=T_q(t)y_0+\int_0^tT_q(t-\tau)F_q(y(\tau))\dd\tau\\
&=T_q(t)P_q^sy_0+\int_0^tT_q(t-\tau)P_q^sF(y(\tau))\dd\tau-\int_t^{\infty}P_q^cF(y(\tau))\dd\tau\\
&\qquad + P_q^cy_0+\int_t^{\infty}P_q^cF(y(\tau))\dd\tau + \int_0^t P_q^cF(y(\tau))\dd\tau,
\end{align*}  
using again Lemma~\ref{Lproper} and (c).
The definition of $\Phi_q(y,z_0)$ in \eqref{eq4.7.1} then yields
\begin{equation}
\label{EQFT}
y(t)=\Phi_q(y,z_0)(t)+z_c,\qquad t\geq 0.
\end{equation}
Due to (c) and \eqref{eq4.7.3}, the functions $y$ and $\Phi_q(y,z_0)$ tend to 0 in $\cE_\alpha$ 
as $t\rightarrow\infty$, and hence $z_c=0$. Equation \eqref{EQFT} thus implies  $y=\Phi_q(y,z_0)$ 
so that (a) is a consequence of the observations after 
\eqref{defnphi}.
\end{proof}

\begin{proof}[Proof of Theorem~\ref{STFOLTH}] Recall from Remark~\ref{DL} that all mild solutions of 
 \eqref{eq4.2.3} are given by $y+Y_q$ for a mild solution $y$ of \eqref{eq4.3.3}.
 
(i) and (ii). Equations \eqref{LPFP} and \eqref{defnphi} show that 
$z_0+\phi_q(z_0)$ is the value of $\Phi(z_0, \phi_q(z_0))$ at $t=0$. From \eqref{y-lip} we then deduce 
that $\phi_q$ and hence $\cM_q^s$ are Lipschitz in  $\cE_\beta=\cE_0\cap\cE_\alpha$.
 
 Let $y_0+Y_q$ belong to $\cM_q^s$, 
 where $z_0=P_q^sy_0\in \Ran(P_q^s)\cap\mathbb{B}_{\delta_0}(|\cdot|_\beta)$. 
By Lemma~\ref{LEQ}, the fixed point $y_{z_0}^q$ is the mild solution $Y(\cdot; y_0)$ of 
 \eqref{eq4.3.3} in $\mathbb{B}_{\delta}(\|\cdot\|)$ of 
 \eqref{eq4.3.3} with the initial value $y_0$.  Combined with \eqref{y-lip} and
 Remark~\ref{Projq}, these facts imply (ii).
 
 Take $t_0>0$ such that $|y(t_0)+Y_q-Y_0|_\beta\le \eta$.  It is easy to see that $y(t_0+\cdot)$
 still belongs to $\mathbb{B}_{\delta}(\|\cdot\|)$ and that it  is the mild solution of 
\eqref{eq4.3.3} with the initial value $y(t_0)$. Moreover, Remark~\ref{Projq} (in particular, that $P_q^s\in\mathcal{B}(\mathcal{E}_\beta)$) and \eqref{y-diff} yield 
\begin{equation}\label{est:delta0}
|P_q^s y(t_0)|_\beta \le  C\,(|y(t_0) + Y_q - Y_0|_\beta + |Y_0 - Y_q|_\beta) 
    \le C(\eta  + q) \le \delta_0,
\end{equation}    
if we choose $\eta>0$ and $q_0$ small enough. (Note that the constants are uniform for $q$ in compact 
intervals and independent of $\eta$.) Therefore, 
$y(t_0)+Y_q$ is contained in $\cM_q^s$ thanks to  Lemma~\ref{LEQ}. So (i) is shown.
 
(iii). Take $Y(0)\in Y_0+\mathbb{B}_{\eta}(|\cdot|_{\beta})$ that satisfies properties  (a)--(c) in item (ii). 
The function $y(t)=Y(t;Y(0))-Y_q$ is a mild solution of \eqref{eq4.3.3} with initial value $Y(0)-Y_q$. Using 
again  \eqref{y-diff}, we can estimate
 \[ |Y(0)-Y_q|_\beta\le |Y(0)-Y_0|_\beta+|Y_q-Y_0|_q\le \eta+ Cq.\]
 Possibly decreasing $\eta,q_0>0$, we deduce from conditions (a)--(c) the inequality  \eqref{eq4.7.3} for $y$ 
and from Remark~\ref{Projq} the estimate $|P_q^s (Y(0)-Y_q)|_\beta\le \delta_0 $.
Lemma~\ref{LEQ} now yields that $y(0)\in\text{\rm{graph}}_{\delta_0}\phi_q$, proving (iii). 

(iv). By Theorem~3.14  in \cite{GLS2}, we can fix a sufficiently small radius $\eta>0$ such that 
for each point $Y(0)$ in the ball $Y_0+\mathbb{B}_{\eta}(|\cdot|_{\beta};Y_0)$ 
there exists a shift $q=q(Y(0))$ such that the solution $Y(\cdot;Y(0))$ of \eqref{eq4.2.3} satisfies 
properties~(a)--(c) of item~(ii). We remark that in Theorem~3.14 we can choose the same number $\delta>0$ 
as in the current proof and exponents\footnote{In (7) of Theorem~3.14 there is a misprint, 
one has to replace $\nu$ by $\rho$.}  
$\nu,\rho>\omega$ which are different from our exponents $\nu$ and $\rho$ in \eqref{eq4.7.4}. 
Item (iii) then implies  that $Y(0)$ is contained in $\mathcal{M}_q^s$. If $Y(0)$ is also an element
of  $\mathcal{M}_{\bar{q}}^s$ for some $\bar{q}\in [-q_0,q_0]$, then the corresponding
solution $y$ would converge both to $Y_q$ and $Y_{\bar{q}}$ as $t\to\infty$, and so $q=\bar{q}$.
Hence, (iv) holds.

(v). Let $|q|,|\bar{q}|\le q_0$ and  $z_0\in\mathbb{B}_{\delta_0}(|\cdot|_{\beta})$. The maps 
$q\mapsto P_q^c\in\cB(\cE_\kappa,\cE_\beta)$, $q\mapsto P_q^s\in\cB(\cE_\kappa)$  and 
$q\mapsto B_q\in \cB(\cE_\alpha,\cE_0)$ are Lipschitz for $\kappa\in\{\beta,\alpha\}$ due to  
\eqref{eq:proj-diff} and Remark~\ref{RB}. Lemma~3.7 of \cite{GLS2} implies that $\gamma_\alpha Y_0'$ and 
 $\gamma_\alpha^{-1} Y_0'$ are bounded. Using \eqref{eq4.3.4} and \eqref{y-diff}, we then deduce the estimate
\[ |F_q(Y)-F_{\bar{q}}(Y)|_\beta\le C |Y|_{\kappa} \,|q-\bar{q}|   \]
for all $Y\in\cE_\kappa$ and $\kappa\in \{0,\alpha\}$.
In view of \eqref{defnphi}, for (v) it remains to check that
the map $q\mapsto y_{z_0}^q=:y_q$ is Lipschitz for $\|\cdot\|$. Since $y_q$ is the fixed point, we 
infer from \eqref{eq4.7.1} the identity
\[ y_q-y_{\bar{q}} = \Phi_q(y_q,z_0)-  \Phi_{\bar{q}}(y_q,z_0) 
    + \Phi_{\bar{q}}(y_q,z_0) -  \Phi_{\bar{q}}(y_{\bar{q}}, z_0).\]
By \eqref{est:phi-lip}, the second difference on the right hand side is bounded by
$C\delta\,\|y_q-y_{\bar{q}}\|$ and can thus be absorbed by the left hand side possibly
after decreasing $\delta>0$ once more. To control the other difference, we note that 
the bounded perturbation theorem and \eqref{diffL}  imply that
$q\mapsto T_q(t)\in\cB(\cE_\kappa)$ is Lipschitz for $\kappa\in \{0,\alpha\}$ and uniformly for $t\ge0$ in 
compact sets, see Corollary~3.1.3 of \cite{pazy}. To extend this property to $\RR_+$, let $t\in(n,n+1]$. We write
\begin{align*}
T_q(t)P_q^s - T_{\bar{q}}(t)P_{\bar{q}}^s
& = (T_q(t-n) - T_{\bar{q}}(t-n)) T_q(n)  P_q^s+ T_{\bar{q}}(t-n) T_q(n) P_q^s(P_q^s-\! P_{\bar{q}}^s) \\
 &  \quad + T_{\bar{q}}(t-n) \sum_{k=0}^{n-1} T_q(n-k-1) P_q^s (T_q(1) - T_{\bar{q}}(1)) 
    T_{\bar{q}}(k)P_{\bar{q}}^s\\
 & \quad  + T_{\bar{q}}(t-n) (P_q^s- P_{\bar{q}}^s) T_{\bar{q}}(n)P_{\bar{q}}^s.
\end{align*}
In the exponential decay estimate \eqref{exp-stab}  for $T_q(t)P_q^s $ we can replace $\nu$ by a 
slightly larger number, see Lemma~3.13 of \cite{GLS2}. This and the above mentioned facts lead to 
the inequality
\[ \|T_q(t)P_q^s - T_{\bar{q}}(t)P_{\bar{q}}^s\|_{\cB(\cE_\alpha)} \le C\e^{-\nu t}\,|q-\bar{q}|, 
   \qquad t\ge0.\]
As in Lemma~\ref{Lproper} one can now show that 
\[\|\Phi_q(z_0,y_q)-  \Phi_{\bar{q}}(z_0,y_q) \|\le C|q-\bar{q}|.\]
Summing up,  (v) is true.
\end{proof}

To conclude, we briefly mention two motivating examples borrowed from \cite{GLS3} that fit our setting. More details can be 
found in the papers
\cite{GLS1} and \cite{GLS2}, respectively. We stress, however,  that for this type of examples the Hypotheses~\ref{HRF}, 
\ref{HRFR} and \ref{HypSpL} (a) can  rigorously be verified not in all cases while the absence of the unstable eigenvalues 
required in Hypothesis \ref{HypSpL} (b) is usually checked only numerically for certain ranges of the parameter values.

\begin{example}\label{exa1} {\em Gasless combustion.}
A simple combustion model in one space dimension has been mentioned in the Introduction and is given by the system
\[
\partial_t u=\partial_{xx}u+v g(u), \quad
\partial_t v=-\beta v g(u), \]
where $ g(u) = e^{-\frac{1}{u}}$  if $u>0$ and $ g(u)=0$  if $u\le 0$.
In this system,  $u$ is the temperature, $v$ is the concentration of unburned fuel, $g$ is the unit reaction rate, 
and  $\beta>0$ is a constant parameter.  This system was a primary guiding example in \cite{G_indiana,GLS,GLS1,GLS2,GLS3}. 
One motivation for looking at this well-studied problem, in which the reactant does not diffuse, was heat-enhanced methods 
of oil recovery in which the reactant is coke contained in the rock formation, see \cite{akkutlu-yorsos03}.  
The value $u=0$ represents the ignition temperature and is also taken to be the background temperature, at which the 
reaction does not take place. 

 Clearly, Hypothesis~\ref{HRFR} is satisfied here. 
One looks for traveling waves $Y_0=(u_0,v_0)$ such that $Y_-=(u_{-},0)$ with $u_{-}>0$, 
$Y_+=(0,1)$, and $(u_0(x),v_0(x))$ approaches these end states exponentially as $x\to\pm\infty$. For each 
$\beta>0$ there is a unique $c>0$ for which such a wave exists, cf.\ \cite[\S 3.2] {GLS3}. This wave represents 
a combustion front that leaves behind 
of it high temperature $u_-=1/\beta$ and no fuel, while in front of it temperature is $0$ and there is fuel, with  concentration normalized to $1$.  As discussed in Paragraph~3.2 of \cite{GLS3}, Hypothesis~\ref{hypos} is true
and Hypothesis~\ref{HypSpL} can be verified (partly numerically) for small $\beta>0$.

We note the lack of diffusion in the second equation which inspired the  linear  Lemma~3.13 in \cite{GLS2} used to derive 
the exponential decay \eqref{exp-stab} from the spectral assumptions in Hypothesis \eqref{HypSpL}, and the form of 
the  nonlinear term in this and related problems which inspired the triangular and product structure of the nonlinearity 
in the current paper that follows from Hypothesis \ref{HRFR}.
\end{example}

\begin{example}\label{exa2}{\em Exothermic-endothermic chemical reactions.}
A model in which two chemical reactions occur at rates determined by temperature  was studied in \cite{SMS,SKMS2}, see 
also \cite{GLS2}.  One reaction is exothermic (produces heat), the other is endothermic (absorbs heat). The system reads
\begin{align}
\partial_t y_1    &=   \partial_{xx}y_1 +  y_2f_2(y_1) -  \sigma y_3f_3(y_1),  \label{simon1} \\
\partial_t y_2    &=   d_2\partial_{xx}y_2  -  y_2f_2(y_1),  \label{simon2} \\
\partial_t y_3    &=   d_3 \partial_{xx}y_3 -   \tau y_3f_3(y_1).   \label{simon3}   
\end{align}
Here $y_1$ is the temperature, $y_2$ is the quantity of an exothermic reactant, and $y_3$ is the quantity of an endothermic 
reactant.  The parameters  $\sigma$ and $\tau$ are positive, and there are positive constants $a_i$ and $b_i$ such that 
$f_i(u)=
a_ie^{-\frac{b_i}{u}}$ for $u>0$ and 
$f_i(u)=0$ for $u \le 0$.
In \cite{SMS,SKMS2} it is shown numerically that in certain parameter regimes there exist
 traveling wave solutions $Y_0$ of \eqref{simon1}--\eqref{simon3} with speed $c>0$ and the end states 
 $Y_-=(1-\frac{\sigma}{\tau},0,0)$  and $Y_+=(0,1,1)$.  Moreover,  both end states are approached at an exponential 
 rate, the zero eigenvalue of the linearization is simple, and there are no other eigenvalues in the right half plane.
A rigorous motivation for the existence of such traveling wave is also given in \cite[Section 9.2]{GLS2}. Assuming the 
existence of the traveling wave with these properties, the remaining hypotheses of the current paper are easy to verify.
\end{example}

\end{document}